\documentclass[11pt,a4paper]{article}

\usepackage[T1]{fontenc}
\usepackage{lmodern}
\usepackage[utf8]{inputenc}

\usepackage[left=2.0cm, top=2.45cm,bottom=2.45cm,right=2.0cm]{geometry}
\usepackage{mathtools,amssymb,amsthm,mathrsfs,calc,graphicx,xcolor,dsfont,tikz,bm}
\usepackage[british]{babel}
\usepackage{amsfonts}              

\usepackage[bbgreekl]{mathbbol}
\usepackage{bbm}
\DeclareSymbolFontAlphabet{\mathbb}{AMSb}

\usepackage[
	bookmarks=true,
	bookmarksnumbered=true,
	bookmarksopen=true,
	unicode=true,
	pdftoolbar=true,
	pdfmenubar=true,
	pdffitwindow=false,
	pdfstartview={FitH},
	pdftitle={},
	pdfauthor={},
	pdfsubject={},
	pdfcreator={},
	pdfproducer={},
	pdfkeywords={},
	pdfnewwindow=true,
	colorlinks=true,
	linkcolor=black,
	citecolor=black,
	filecolor=black,
	urlcolor=black]{hyperref}

\usepackage{cleveref}
\usepackage[colorinlistoftodos]{todonotes}

\usepackage[T1]{fontenc}
\usepackage[deletedmarkup=xout]{changes}
\usepackage{selinput}
\numberwithin{equation}{section}
\numberwithin{figure}{section}

\makeatletter
	\def\@cite#1#2{[\textbf{#1}\if@tempswa , #2\fi]}	
	\def\@biblabel#1{[#1]}								
\makeatother

\newtheorem {theorem}{Theorem}[section]
\newtheorem {proposition}[theorem]{Proposition}
\newtheorem {lemma}[theorem]{Lemma}
\newtheorem {corollary}[theorem]{Corollary}

\newtheorem {remark}[theorem]{Remark}

\newtheorem {problem}[theorem]{Open Problem}

\newtheorem {conjecture}[theorem]{Conjecture}

\newcommand{\bd}{\operatorname{bd}}

\newcommand{\dint}{\textup{d}}

 \def\EE{\mathbb{E}}
 \def\NN{\mathbb{N}}
 \def\PP{\mathbb{P}}
 \def\RR{\mathbb{R}}
 \def\SS{\mathbb{S}}
 \def\bZ{\mathbf{Z}}

\def\bbeta{{\boldsymbol{\beta}}}

\def\bX{\mathbf{X}}
\def\bY{\mathbf{Y}}

 \def\bd{\mathbf{d}}

 \def\br{\mathbf{r}}

 \def\bu{\mathbf{u}}
 \def\by{\mathbf{y}}
 \def\bz{\mathbf{z}}
 \def\bx{\mathbf{x}}
 \def\bw{\mathbf{w}}
 \def\bv{\mathbf{v}}
 \def\be{\mathbf{e}}

\def\cP{\mathcal{P}}
\def\cF{\mathcal{F}}

\def\bone{\mathbbm{1}}
\def\bd{\mathbbm{d}}
\def\bbeta{\bbbeta}

\newcommand{\proj}{\operatorname{proj}}
\newcommand{\vol}{\operatorname{Vol}}

\makeatletter
\let\@fnsymbol\@alph
\makeatother

\begin{document}

\title{\bfseries Random polytopes in convex bodies:\\ Bridging the gap between extremal containers
}

\author{Florian Besau\footnotemark[1],\; Anna Gusakova\footnotemark[2],\; Christoph Th\"ale\footnotemark[3]}

\renewcommand{\thefootnote}{\fnsymbol{footnote}}

\footnotetext[1]{
    Technische Universität Wien, Austria. Email: florian@besau.xyz
}
\footnotetext[2]{
    M{\"u}nster University, Germany. Email: gusakova@uni-muenster.de
}
\footnotetext[3]{%
    Ruhr University Bochum, Germany. Email: christoph.thaele@rub.de
}

\date{}

\maketitle

\begin{abstract}
We investigate the asymptotic properties of random polytopes arising as convex hulls of $n$ independent random points sampled from a family of block-beta distributions. Notably, this family includes the uniform distribution on a product of Euclidean balls of varying dimensions as a key example. As $n\to\infty$, we establish explicit growth rates for the expected number of facets, which depend in a subtle way on the the underlying model parameters. For the case of the uniform distribution, we further examine the expected number of faces of arbitrary dimensions as well as the volume difference. Our findings reveal that the family of random polytopes we introduce exhibits novel interpolative properties, bridging the gap between the classical extremal cases observed in the behavior of random polytopes within smooth versus polytopal convex containers.


\end{abstract}

\section{Introduction and results}

\subsection{Motivation and background}

Random polytopes, arising as convex hulls of a finite but large set of randomly chosen points in a given convex container, play a significant role in numerous fields of mathematics and applied sciences. Their study lies at the intersection of probability theory, geometric analysis, integral geometry, combinatorics, and approximation theory. Beyond their theoretical appeal, random polytopes play significant role in various applications in areas such as optimization, computational geometry, statistical learning, and data science, where they used to model and solve complex problems related to shape approximation, volume estimation, and geometric inference. 

A primary motivation for investigating random polytopes is to understand and model random structures and processes in high-dimensional spaces, which are increasingly common in modern scientific inquiries. In machine learning, for example, random polytopes help approximate solution spaces in optimization problems, while in data analysis, they provide insights into the geometry of data distributions. Furthermore, the probabilistic properties of random polytopes -- such as their volume, surface area, and the number of faces -- are of great interesting since they reveal fundamental properties of random geometric objects.
\smallskip

In the literature, two set-ups are typically studied:
\begin{itemize}
    \item[(i)] random polytopes generated by random points in convex bodies with smooth boundary (meaning that the boundary is twice differentiable and has positive Gauss curvature at every boundary point),
    \item[(ii)] random polytopes in convex polytopes. 
\end{itemize}
These two cases exhibit radically different behaviors. For instance, if $K\subset\RR^d$ with $\vol_d(K)=1$ is a convex body with smooth boundary and $K_n$ is the convex hull of $n$ independent and uniform random points in $K$, then 
 by \cite[Thm.\ 4]{Reitzner2005} the expected number of $j$-dimensional faces $f_j(K_n)$ of $K_n$ behaves asymptotically as
\begin{equation}\label{eq:Expectation f_j smooth}
    \EE f_j(K_n) 
    = c_{d,j}^{\rm sm}\,\Omega(K)\,n^{\frac{d-1}{d+1}}(1+o_n(1)), \qquad\qquad j\in\{0,\ldots,d-1\},
\end{equation}
 where $c_{d,j}^{\rm sm}\in(0,\infty)$ is a constant only depending on $d$ and $j$, $o_n(1)$ indicates a sequence that tends to zero with $n\to\infty$. Here, $\Omega(K)$, defined by  
 \begin{equation*}
    \Omega(K)
    :=\int_{\partial K}\kappa(K;\bx)^{\frac{1}{d+1}}\,\mathcal{H}^{d-1}(\dint \bx),
 \end{equation*}
is Blaschke's classical notion of \textit{affine surface area} of $K$, see \cite{LR:1999, Lutwak:1991, SchuttWernerSurvey}. In this expression, $\kappa(K;\bx)$ denotes the Gauss curvature of the hypersurface $\partial K$ at a point $\bx\in\partial K$, and $\mathcal{H}^{d-1}$ is the $(d-1)$-dimensional Hausdorff measure.

In contrast, if $P\subset\RR^d$ is a polytope with $\vol_d(P)=1$, and $P_n$ is the convex hull of $n$ independent and uniform random points in $P$, then
\begin{equation}\label{eq:Expectation f_j polytope}
    \EE f_j(P_n) 
    = c_{d,j}^{\rm po}\,T(P)\,(\ln n)^{d-1}(1+o_n(1)), \qquad\qquad j\in\{0,\ldots,d-1\},
\end{equation}
by \cite[Thm.\ 8]{Reitzner2005}, where $c_{d,j}^{\rm po}\in(0,\infty)$ is another constant only depending on $d$ and $j$, and $T(P)$ is the \textit{number of complete flags} of the polytope $P$. A complete flag of $P$ is a sequences $F_0\subset F_1\subset\ldots\subset F_{d-1}$, where each $F_i$, $i\in\{1,\ldots,d-1\}$, is a face of $P$ of dimension $i$. 

Similarly, for the expected volume difference, we have
\begin{align}
    \label{eq:Expectation Vol smooth} 
        \EE \vol_d(K \setminus K_n) 
        &= c_d^{\rm sm}\,\Omega(K)\,n^{-\frac{2}{d+1}}(1+o_n(1)),\\
    \label{eq:Expectation Vol polytope} 
        \EE \vol_d(P\setminus P_n) 
        &= c_d^{\rm po}\,T(P)\,\frac{(\ln n)^{d-1}}{n}(1+o_n(1)),
\end{align}
with constants $c_d^{\rm sm},c_d^{\rm po}\in(0,\infty)$ only depending on $d$. This result follows from \eqref{eq:Expectation f_j smooth} and \eqref{eq:Expectation f_j polytope} together with Efron's identity \cite[Eqn.\ (3.7)]{Efron65}, which implies that for general convex bodies $L\subset\RR^d$ with $\vol_d(L)=1$, $\EE f_0(L_n)$ and $n(\vol_d(L)-\EE\vol_d(L_n))$ are asymptotically equivalent as $n\to\infty$, where $L_n$ denotes the convex hull of $n$ independent random points uniformly distributed in $L$. 
\smallskip

The asymptotic behavior in \eqref{eq:Expectation f_j smooth}--\eqref{eq:Expectation Vol polytope} highlights a general pattern: for containers with smooth boundary, the expected number of faces behaves asymptotically like a power of $n$, whereas for polytopal containers, it grows logarithmically.
For additional details on these and related results, we refer to the survey articles \cite{BaranySurvey, HugSurvey, PSW:2022, ReitznerSurvey, SchneiderSurvey} and the references therein. 
\medskip

Given the dichotomy between the smooth and the polytopal case, a natural question arises: what is the expected behavior of random convex hulls in a `generic' convex body? B\'ar\'any and Larman \cite[Thm.\ 5]{BaranyLarman88} demonstrated that, in the sense of Baire category, the behavior of $\EE f_j(L_n)$ and $\EE\vol_d(L_n)$ for most convex bodies $L\subset\RR^d$ is unpredictable. Specifically, for infinitely many $n$,
\begin{align*}
    \EE\vol_d(L\setminus L_n) &< \frac{(\ln n)^{d-1}}{n}\,a_n,
    \intertext{and,}
    \EE\vol_d(L\setminus L_n) &> n^{-\frac{2}{d+1}}\,b_n,
\end{align*}
where $(a_n)_{n\geq 1}$ and $(b_n)_{n\geq 1}$ are arbitrary sequences of positive real numbers with $a_n\to\infty$ and $b_n\to 0$ as $n\to\infty$. A similar oscillating behavior also holds for $\EE f_j(L_n)$, see \cite[Cor.\ 3]{Barany1989}. In other words, for generic containers, the expected volume and the expected face numbers oscillates between the smooth and the polytopal case. 

This raises a major open question in the field: are there `natural' classes of convex bodies which do not exhibit such a chaotic behavior and interpolate between \eqref{eq:Expectation f_j smooth} and \eqref{eq:Expectation f_j polytope}, or between \eqref{eq:Expectation Vol smooth} and \eqref{eq:Expectation Vol polytope} respectively?
The principal aim of this article is to address this question by introducing a family of convex bodies that serves as an interpolating class, which displays the growth rate of both extremal cases in a unified framework. 
\medskip

The asymptotic results in \eqref{eq:Expectation f_j smooth}--\eqref{eq:Expectation Vol polytope} and the question posed above are closely related to the geometric properties of the floating bodies associated with the container body in which the random polytope is constructed. We recall that the (convex) \textit{floating body} $L_{[\delta]}$, $\delta>0$, of a convex body $L\subset\RR^d$ arises from $L$ by removing all caps of $L$ with volume at most $\delta$, see the survey article \cite{SchuttWernerSurvey}. More precisely, $L_{[\delta]}=\bigcap\{H^+:\vol_d(H^-\cap L)\leq\delta\}$, where $H^\pm$ stand for the two half-spaces determined by a hyperplane $H$. Then, if $L_n$ is the convex hull of $n$ independent and uniform random points in $L$ with $\vol_d(L)=1$, it is known from \cite[Thm.\ 1]{BaranyLarman88} and \cite[Thm.\ 1]{Barany1989} that 
\begin{align}
    \label{eq:FloatingBody for fj} c\vol_d(L\setminus L_{[1/n]}) &\leq \EE\vol_d(L\setminus L_n) \leq c_d\vol_d(L\setminus L_{[1/n]}),
    \intertext{and}
    \label{eq:FloatingBody for Vol} c\,n\vol_d(L\setminus L_{[1/n]}) &\leq \EE f_j(L_n) \leq c_d\,n\vol_d(L\setminus L_{[1/n]}),
\end{align}
for all $n\geq N_d$, where $c\in(0,\infty)$ is an absolute constant and $c_d,N_d\in(0,\infty)$ are constants only depending on $d$. 

However, there are only few classes of convex bodies $L$ for which the volume $\vol_d(L\setminus L_{[\delta]})$ is explicitly known or even for which the asymptotic order as $\delta\downarrow 0$ can be determined. This behavior is primarily understood for convex bodies with sufficiently smooth boundary \cite{BLW:2018,SW:1990} and for polytopes \cite{BSW:2018,Schutt:1991}. Notably, there are also some special planar examples constructed in \cite{SchuttWerner1992}. 

These observations reveal that the relations in \eqref{eq:FloatingBody for fj} and \eqref{eq:FloatingBody for Vol} between floating bodies and random polytopes are insufficient to address our question. Rather, the reverse approach will hold: via \eqref{eq:FloatingBody for fj} and \eqref{eq:FloatingBody for Vol}, our results on random polytopes will introduce a new class of convex bodies $L$ for which the asymptotic behavior of $\vol_d(L\setminus L_{[\delta]})$ as $\delta\downarrow 0$ can now be analyzed.

\begin{figure}[t]
    \centering
    \includegraphics[width=\linewidth]{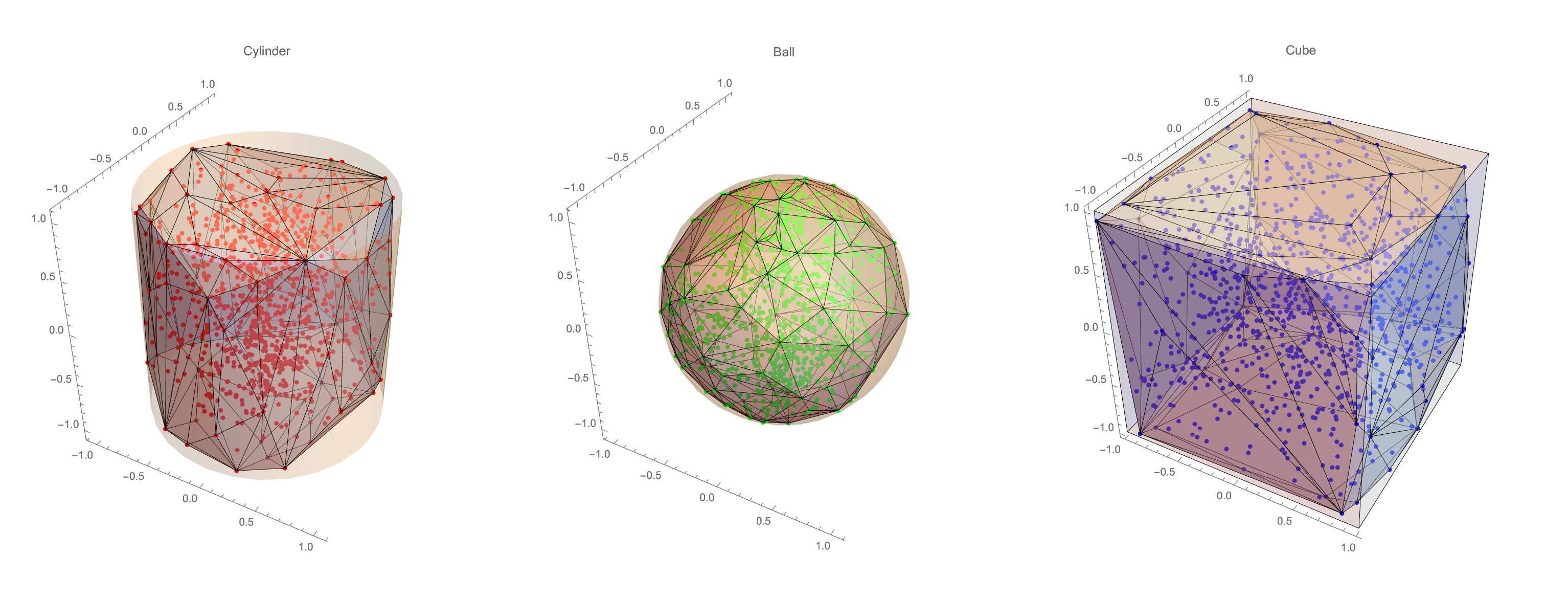}
    \caption{Simulations of uniform random polytopes generated by $n=1\,000$ points in a cylinder $\bd=(2,1)$ (left), the unit ball $\bd=(3)$ (middle) and the cube $\bd=(1,1,1)$ (right).}
    \label{fig:Simulations}
\end{figure}

\subsection{Principal results}

To introduce our setup, let $d\geq 2$ and $m\in\{1,\ldots,d\}$ be integers, and consider an $m$-tuple $\bd:=(d_1,\dotsc,d_m)\in \NN^m$, such that $d_1+\ldots+d_m = d$. We define the origin symmetric convex body $Z_{\bd}$ as the product
\begin{equation}\label{eq:DefZd}
	Z_{\bd} := \prod_{i=1}^m B_2^{d_i} \subset \RR^d,
\end{equation}
of $d_i$-dimensional centered Euclidean unit balls $B_2^{d_i}$, $i\in\{1,\ldots,m\}$. Throughout the article, we use the natural identification of the $d_i$-dimensional linear subspace $\{0\}^{d_1+\dotsc+d_{i-1}}\times\RR^{d_i}\times \{0\}^{d_{i+1}+\dotsc+d_m}$ with $\RR^{d_i}$. Alternatively, $Z_{\bd}$ can be described as the unit ball in $\RR^d$ given by the norm
\begin{equation*}
	\|(\bx_1,\dotsc,\bx_m)\|_{Z_{\bd}} := \max_{i\in\{1,\dotsc,m\}} \|\bx_i\|_2 \qquad \text{for $\bx_i\in\RR^{d_i}$}.
\end{equation*}
As a Minkowski sum of Euclidean balls, the convex body $Z_{\bd}$ is also a zonoid, but in general not a zonotope. Let us discuss some special cases:
\begin{itemize}
    \item[(i)] $m=1$: Here, $\bd=(d)$, and $Z_{\bd}$ reduces to the $d$-dimensional Euclidean unit ball $B_2^d$. This case represents the classical smooth container, which has been studied extensively, for example in \cite{Affentranger:1991, BH:2024, BHK:2021, BM:1984, Kur:2020,Muller:1990}.
    
    \item[(ii)] $m=d$: For $\bd=(1,\dotsc,1)$, $Z_{\bd}$ is the centered cube $B_{\infty}^d=[-1,1]^d$. As a Minkowski sum of line segments, this body is a zonotope. This case represents the classical case of polytopal containers studied, for example, in \cite{AW91,Barany1989,BaranyBuchta,BSW:2018,GusakovaReitznerThaele,Reitzner2005,ReitznerSchuettWernerPolytopeBoundaryPolytope,RenyiSulanke1,Schutt:1991}.
    
    \item[(iii)] $m=2$: If $\bd=(k,d-k)$ for some $k\in\{1,\ldots,d-1\}$, then $Z_{\bd} = B_2^{k}\times B_2^{d-k}$. For example, $Z_{(d-1,1)}=B_2^{d-1}\times [-1,1]$ is the $d$-dimensional cylinder of height $2$ having a $(d-1)$-dimensional Euclidean unit ball as base. A special case will be played by \textit{Lagrangian produces} of balls. They arise if $d=2k$ is even and $Z_{(k,k)} = B_2^{k}\times B_2^k$. We remark that Lagrangian products of convex bodies have recently played a significant role in the construction of the counterexample to the famous Viterbo conjecture for capacities in symplectic geometry \cite{HaimKislevOstrover2024}.
\end{itemize}

In this paper we are considering the \textit{random uniform polytope}
\begin{equation*}
    \cP_{n,{\bd}} := \mathrm{co}\,\{\bX_1,\dotsc,\bX_n\},
\end{equation*}
where $\bX_1,\dotsc,\bX_n$ are independent random points chosen uniformly in $Z_{\bd}$.
More generally, we also consider the \textit{random block-beta polytope}
\begin{equation*}
    \cP_{n,{\bd}}^{\bbeta}:=\mathrm{co}\,\{\bX_1,\dotsc,\bX_n\},
\end{equation*}
where the generating points $\bX_i$, $i\in\{1,\ldots,n\}$, are distributed according to a \textit{block-beta distribution} on $Z_{\bd}$ with block-parameter $\bbeta:=(\beta_1,\dotsc,\beta_m)$, where $\beta_1,\ldots,\beta_m>-1$.  See Figure \ref{fig:Simulations} for some examples of these random polytopes for different containers in $\RR^3$.

To define our setup more rigorously, for each $i\in\{1,\ldots,n\}$ we group the coordinates $X_{i,1},\ldots,X_{i,d}$ of $\bX_i$ according to the block structure of $Z_\bd$ as $\bX_i=(\bX_{i}^{(1)},\ldots,\bX_{i}^{(m)})$ with $\bX_{i}^{(j)}=(X_{i,\ell_j+1},\ldots,X_{i,\ell_j+d_j})$ and where $\ell_j:=d_1+\ldots+d_{j-1}$ for $j\in\{1,\ldots,m\}$. Then, we require that for each $i\in\{1,\ldots,n\}$ and $j\in\{1,\ldots,m\}$ the block-coordinates $\bX_{i}^{(j)}$ are independent and beta-distributed on $\RR^{d_j}$ with parameter $\beta_j$, that is, $\bX_{i}^{(j)}$ has probability density
\begin{equation}\label{eqn:beta_density}
    f_{\beta_j,d_j}(\by) = \frac{\Gamma(\frac{d_j}{2}+\beta_j+1)}{\pi^{\frac{d}{2}}\Gamma(\beta_j+1)} (1-\|\by\|_2^2)^{\beta_j} \qquad \text{for $\by\in B_2^{d_j}$}.
\end{equation}
Notably, if $\beta_1=\ldots=\beta_m=0$, then the generating random points $\bX_i$ are uniform in $Z_\bd$. 

Let us recall from \cite[Lem.\ 4.4]{KTT19Beta} that the beta distribution on a Euclidean ball $B_2^k$ with parameter $\beta=\frac{n-k}{2}\geq 0$ arises as the projection of the uniform distribution on $B_2^n$ under orthogonal projection onto $\RR^k$. Since such projections will play an important role in our investigation, working with beta distributions from the beginning is a natural choice in this context. 

\smallskip
To present our main result, we set, for $\beta_1,\dotsc,\beta_m\geq 0$, 
\begin{equation*}
    d:=\sum_{i=1}^m d_i \qquad \text{ and }\qquad \beta:=\sum_{i=1}^m \beta_i \geq 0,
\end{equation*}
and introduce the {\em beta-adjusted dimensions}
\begin{equation*}
    k_i := \frac{d_i+\beta_i}{1+\beta_i},\qquad i\in\{1,\ldots,m\}.
\end{equation*}
Since $\beta_i\geq 0$ we have $1\leq k_i\leq d_i$ and $k_i=1$ if and only if $d_i=1$.  

\smallskip
Throughout this article we use the notation $a_n\asymp b_n$ for two sequences $(a_n)_{n\geq 1}$ and $(b_n)_{n\geq 1}$ of positive real numbers to indicate that there are constants $c,C\in (0,\infty)$ only depending on the model parameters $\bd$ and $\bbeta$ such that $c\leq \liminf\limits_{n\to\infty}\frac{a_n}{b_n}\leq\limsup\limits_{n\to\infty}\frac{a_n}{b_n}\leq C$. 

\smallskip
We are now prepared to present the main result of this paper. This result describes the precise asymptotic growth as $n\to\infty$ of the expected number of facets of a random block-beta polytope in terms of the maximal beta-adjusted dimension. 

\begin{theorem}\label{thm:main}
Fix integers $d\geq 2$ and $m\in\{1,\ldots,d\}$, and an $m$-tuple $\bd=(d_1,\dotsc,d_m)\in \NN^m$, such that $d_1+\ldots+d_m = d$. Further, let $\beta_1,\ldots,\beta_m\geq 0$ and set $\bbeta:=(\beta_1,\ldots,\beta_m)$. We consider the random block-beta polytope $\cP_{n,\bd}^\bbeta$ and define $k_{\max} := \max\limits_{i=1,\dotsc,m} k_i$ and $\#k_{\max} := \#\left\{i: k_i = k_{\max}\right\}$. 
Then,
\begin{equation*}
    \EE f_{d-1}(\cP_{n,\bd}^{\bbeta}) \asymp n^{\frac{k_{\max}-1}{k_{\max}+1}} (\ln n)^{\#k_{\max}-1}.
\end{equation*}
\end{theorem}

To see that the family of random polytopes we constructed answers the question raised earlier, let us revisit the special cases (i)--(iii) outlined above. We also refer to Table \ref{tab:Order} for examples in small space dimensions and Figure \ref{fig:simulationR4} for a simulation study in $\RR^4$.
\begin{itemize}
    \item[(i)] $m=1$: For $\bd=(d)$ and $\bbeta=(\beta)$ for some $\beta\geq 0$, we have $k_{\max}=\frac{d+\beta}{1+\beta}$ and $\#k_{\max}=1$. In this case, Theorem \ref{thm:main} shows that $\EE f_{d-1}(\cP_{n,\bd}^\bbeta)\asymp n^{\frac{d-1}{d+2\beta+1}}$, recovering thereby the result obtained in \cite[Thm.\ 1.18]{KTZBetaPolytopes2020}. In particular, taking $\beta=0$, we obtain the asymptotic growth \eqref{eq:Expectation f_j smooth} for random convex hulls in convex bodies with smooth boundary.

    \item[(ii)] $m=d$: Then $\bd=(1,\ldots,1)$, and we have $k_{\max}=1$ and $\#k_{\max}=d$. Thus, $\EE f_{d-1}(\cP_{n,\bd}^\bbeta)\asymp(\ln n)^{d-1}$, independently of the choice of the beta parameters $\beta_1,\ldots,\beta_d\geq 0$. This recovers the asymptotic rate given by \eqref{eq:Expectation f_j polytope} for random convex hulls in polytopes.

    \item[(iii)] $m=2$: Here we have $\bd=(k,d-k)$ for some $k\in\{1,\ldots,d-1\}$. Suppose that $\beta_1=\beta_2=0$. Then $k_{\max}=\max\{k,d-k\}$ and $\#k_{\max}=2$ if $2k=d$, while $\#k_{\max}=1$ otherwise. Thus, we have
    \begin{equation*}
    \EE f_{d-1}(\cP_{n,\bd}) \asymp \begin{cases}
        n^{\frac{\max\{k,d-k\}-1}{\max\{k,d-k\}+1}}\,\ln n & \text{if $2k=d$,}\\
        n^{\frac{\max\{k,d-k\}-1}{\max\{k,d-k\}+1}} & \text{otherwise}.
    \end{cases}
    \end{equation*}
\end{itemize}

The example in case (iii) shows that in even space dimensions Lagrangian products of balls play a special role, leading to the emergence of an additional logarithmic factor. We shall now explain this phenomenon in more detail.

At first, Theorem \ref{thm:main} and its proof reveal that if the product body $Z_\bd$ has precisely one factor whose beta-adjusted dimension dominates all the others, the facets of $\cP_{n,\bd}^\bbeta$ will eventually cluster in this part of $Z_\bd$, whose local behavior is that of a smooth convex body. The number of facets lying in other parts of $Z_\bd$ are of lower order. Conversely, if there is more than one factor with maximal beta-adjusted dimension, the order of $\EE f_{d-1}(\cP_{n,\bd}^{\bbeta})$ increases by an additional logarithmic factor.
Intuitively, this logarithmic factor corresponds to the number of surplus facets that are generated by connecting points of two (or more) of different clusters of points, concentrating in the parts of $Z_\bd$ of maximal beta-adjusted dimension. The power of logarithm reflects the number of relevant clusters, or, equivalently, the number of maximal beta-adjusted dimensions of $Z_\bd$. 

\begin{figure}[t]
    \centering
    \includegraphics[width=0.7\linewidth]{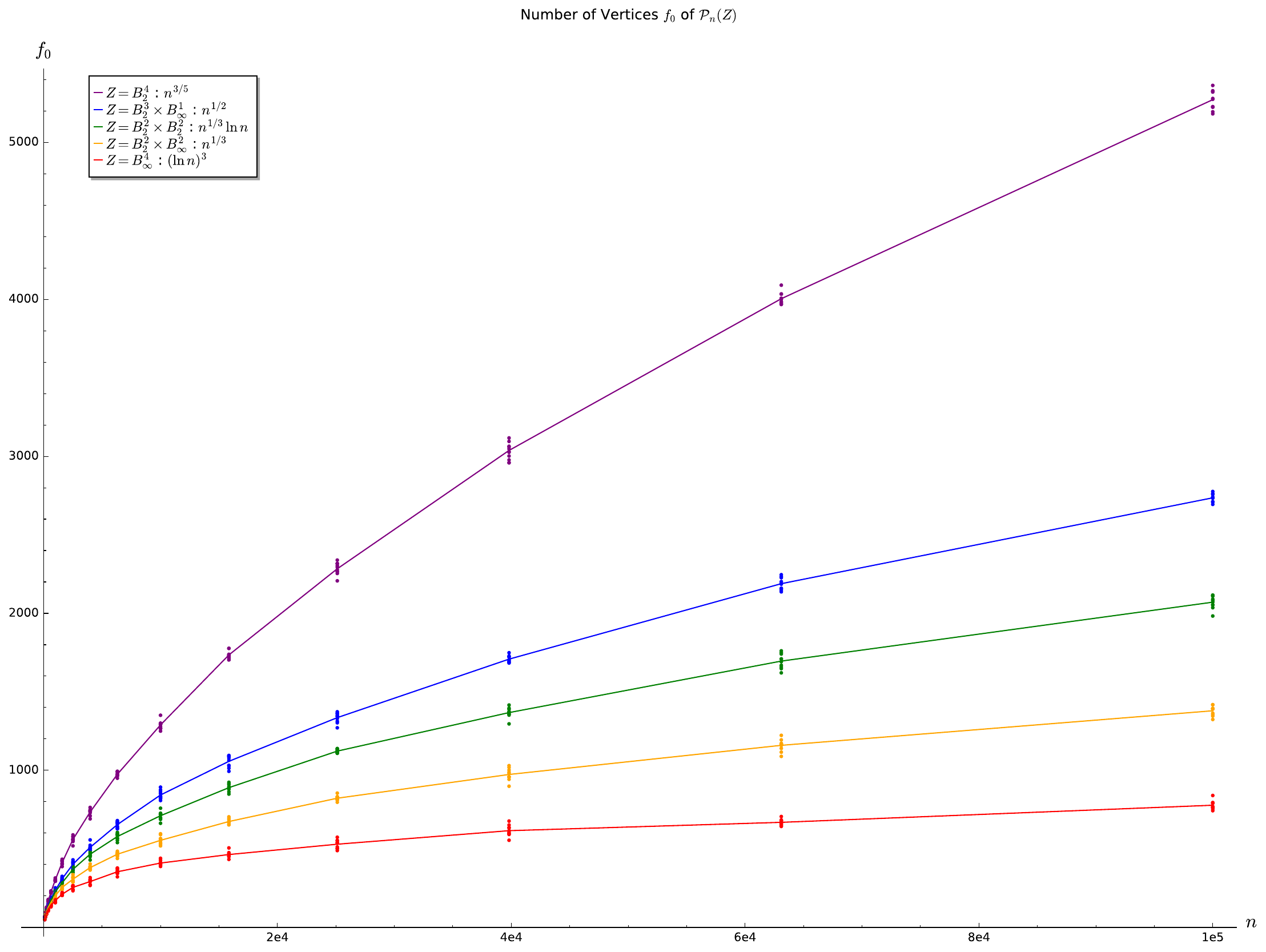}
    \caption{Numerical simulation for $f_{0}(\cP_{n,\bd})$ with $n\leq 10^5$ and with $10$ repetitions for each step for the five different containers cases that can occur in $\RR^4$ from top to bottom: the ball $Z_{(4)}=B_2^4$ in purple with growth rate $n^{3/5}$, the cylinder $Z_{(3,1)}=B_2^3\times B_{\infty}^1$ in blue with growth rate $n^{1/2}$, the Lagrangian product $Z_{(2,2)}=B_2^2\times B_2^2$ in green with growth rate $n^{1/3}\ln n$, $Z_{(2,1,1)}=B_2^2\times B_{\infty}^2$ in yellow with growth rate $n^{1/3}$, and the cube $Z_{(1,1,1,1)}=B_{\infty}^4$ in red with growth rate $(\ln n)^3$.}
    \label{fig:simulationR4}
\end{figure}

For example, if we consider block-beta polytopes in the Lagrangian product $B_2^k\times B_2^k$ in $\RR^{d}$ with $d=2k$, the main contribution to $\EE f_{d-1}(\cP_{n,\bd}^{\bbeta})$ comes from facets that are located close to the $2(k-1)$-dimensional `ridge' submanifold $\SS^{k-1}\times\SS^{k-1}\subset\RR^d$, arising as the intersection of the two `mantle' hypersurfaces $\SS^{k-1}\times B^k_2$ and $B^k_2\times\SS^{k-1}$ whose union give the boundary of $B^k_2\times B^k_2$. 

The striking phenomenon just described is reminiscent of the situation encountered in \cite{ReitznerSchuettWernerPolytopeBoundaryPolytope}, where the authors considered the convex hull of random points uniformly distributed on the boundary of a simple polytope. They showed that in their setup the main contribution to the expected number of facets comes from those facets which are close to the ridges of the container polytope, that is, the intersections of two neighboring facets. 

Additionally, it is noteworthy that the random polytopes we consider fall outside the usual scheme which distinguishes between smooth and polytopal containers as described earlier. The convex containers under consideration are in general not polytopes. However, they still have singular submanifolds in their boundaries around which the facets of the random polytope cluster. This leads to new asymptotic growth rates for $\EE f_{d-1}(\cP_{n,\bd}^{\bbeta})$ in which powers of $n$ \textit{and} logarithmic factors occur at the same time.

\medskip
In view of the relations \eqref{eq:FloatingBody for fj} and \eqref{eq:FloatingBody for Vol} we can extend the result of Theorem \ref{thm:main} to the expected number of faces of arbitrary dimension as well as to the expected volume in the special case of the uniform distribution, which corresponds to the choice $\beta_1=\ldots=\beta_m=0$. We note that in this case the beta-adjusted dimensions satisfy $k_i=d_i$ for $i\in\{1,\ldots,m\}$. However, for face numbers of sufficiently large dimension an extension to general block-beta distributions is possible, see also Remark \ref{rem:Fallj=0}.

\pagebreak
\begin{corollary}\label{cor:main}
Under the assumptions of Theorem \ref{thm:main} the following holds:
\begin{itemize}
    \item[(A)] We have 
    \begin{equation*}
        \EE f_{j}(\cP_{n,\bd}^\bbeta) \asymp n^{\frac{k_{\max}-1}{k_{\max}+1}} (\ln n)^{\#k_{\max}-1},
    \end{equation*}
    for all $j\in\{\lfloor{d/2}\rfloor -1,\ldots,d-1\}$.

    \item[(B)] In the uniform case $\bbeta=(0,\ldots,0)$, we have
    \begin{align*}
        \EE f_{j}(\cP_{n,\bd}^\bbeta) &\asymp n^{\frac{d_{\max}-1}{d_{\max}+1}} (\ln n)^{\#d_{\max}-1},
    \intertext{for all $j\in\{0,\ldots,d-1\}$ and}
        \vol_d(Z_\bd) - \EE\vol_d(\cP_{n,\bd}) &\asymp n^{-\frac{2}{d_{\max} + 1}} (\ln n)^{\#d_{\max}-1},
    \end{align*}
    with $d_{\max} := \max\limits_{i=1,\dotsc,m} d_i$ and $\#d_{\max} := \#\left\{i: d_i = d_{\max}\right\}$.
\end{itemize}
\end{corollary}

\begin{table}[t]
    \centering
    \renewcommand{\arraystretch}{1.5}
    \begin{tabular}{c||c|c||c|c|c||c|c|c|c|c}
         $\bd=$ & $(2)$ & $(1,1)$ & $(3)$ & $(2,1)$ & $(1,1,1)$ & $(4)$ & $(3,1)$ & $(2,2)$ & $(2,1,1)$ & $(1,1,1,1)$\\
         \hline
         $\EE f_{j}(\mathcal{P}_{n,\bd})\asymp$ & $n^{\frac{1}{3}}$ & $\ln n$ & $n^{\frac{1}{2}}$ & $n^{\frac{1}{3}}$ & $(\ln n)^2$ & $n^{\frac{3}{5}}$ & $n^{\frac{1}{2}}$ & $n^{\frac{1}{3}}\ln n$ & $n^{\frac{1}{3}}$ & $(\ln n)^3$
    \end{tabular}
    \caption{The asymptotic order in $n$ for $\EE f_{j}(\mathcal{P}_{n,\bd})$ for particular examples of $\bd$ in small space dimensions.}
    \label{tab:Order}
\end{table}

The remaining sections of this article are dedicated to establishing our Main Theorem \ref{thm:main}. The initial preparation is explained in Subsection \ref{sec:setup} and builds upon classical ideas going back to the seminal work of Rényi and Sulanke \cite{RenyiSulanke1, RenyiSulanke2}. We therefore recall some classical notation and outline some essential background material in Section \ref{sec:Preliminaries}. Notably, in Lemma \ref{lm:AffentrangerWieacker}, we also present a proof for a slightly generalized version of a lemma due to Affentranger and Wieacker \cite{AW91}, concerning the asymptotic behavior of certain product integrals. This is extension is crucial in the latter part of our proof, to determine the asymptotic order as described in our main theorem.

A key geometric insight is gained by investigating the symmetries of our container body $Z_{\bd}$. This enables us to reduce many problems from $Z_{\bd}$ to the $m$-dimensional cube $B_\infty^m:=[-1,1]^m$, which we refer to as \emph{meta-cube}. In Section \ref{sec:GeometricIngredients} we explain this new reduction to the meta-cube in detail alongside with necessary geometric estimates. Following this, we carry out the proof of our Main Theorem \ref{thm:main}, as well as Corollary \ref{cor:main}, in Section \ref{sec:Proofs}. Finally, we conclude the paper in Section \ref{sec:Discussion} with additional remarks on our investigations and a collection of open problems for future exploration.

\section{Preliminaries}\label{sec:Preliminaries}

In the following we recall some notation and background that will  mostly be used for the initial preparation of the proof of our Main Theorem \ref{thm:main} in Subsection \ref{sec:setup}. For a general background on asymptotic geometric analysis and convex geometry we refer to the books \cite{AAGM:2015, Gardner:2006, Schneider:2014}.

\subsection{Notation}

By $\RR$ we denote the set of real numbers and also use the notation $\RR_+:=[0,\infty)$. We usually work in the $d$-dimensional Euclidean space $\RR^d$, which is supplied with the Euclidean norm $\|\bx\|_2$ and the Euclidean scalar product $\bx\cdot\by$, where $\bx,\by\in\RR^d$. The unit ball and the unit sphere in $\RR^d$ are denoted by $B^d_2:=\{\bu\in \RR^{d} : \|\bu\|_2 \leq 1\}$ and $\SS^{d-1} := \{\bu\in \RR^{d} : \|\bu\|_2 = 1\}$, respectively, and we denote by $\sigma_{d-1}$ the spherical Lebesgue measure on $\SS^{d-1}$. The total spherical Lebesgue measure of $\SS^{d-1}$ will be abbreviated with $\omega_d:=\sigma_{d-1}(\SS^{d-1})={2\pi^{d/2}\over\Gamma(d/2)}$. The convex hull of a set $A\subset\RR^d$ will be denoted by ${\rm co}\, A$. Moreover, we write $H(\bw,s):=\{\bx\in \RR^d: \bx\cdot \bw=s\}$ for a hyperplane with normal direction $\bw\in\SS^{d-1}$ and (signed) distance $s\in \RR$ to the origin and $H^+(\bw,s):=\{\bx\in \RR^d: \bx\cdot \bw\ge s\}$, $H^-(\bw,s):=\{\bx\in \RR^d: \bx\cdot \bw\leq s\}$ for the corresponding closed half-spaces.

For a vector $\bw\in\RR^d$ we write $w_1,\ldots,w_d$ for its coordinates. On the other hand, if we identify $\RR^d$ with the block space $\RR^{d_1}\times\cdots\times\RR^{d_m}$ with $m\in\NN$, $d_1,\ldots,d_m\geq 1$ and $d_1+\ldots+d_m=d$, we write $\bw^{(1)}\in\RR^{d_1},\ldots,\bw^{(m)}\in\RR^{d_m}$ for the block coordinates of the vector $\bw$. In addition, if the individual coordinates of the block coordinates of $\bw$ are relevant, we indicate them by $w_{i,j}$, where $i\in\{1,\ldots,m\}$ and $j\in\{1,\dots,d_i\}$.


Throughout the article we write $a_n\lesssim b_n$, respectively $a_n \gtrsim b_n$ for two sequences $(a_n)_{n\geq 1}$ and $(b_n)_{n\geq 1}$ of positive real numbers to indicate that there is a constant $c\in(0,\infty)$ only depending on the model parameters $\bd$ and $\bbeta$ such that $\limsup\limits_{n\to\infty}{a_n\over b_n}\leq c$, respectively $c\leq \liminf\limits_{n\to\infty}{a_n\over b_n}$. We have $a_n\asymp b_n$ if and only if $a_n\lesssim b_n$ and $a_n\gtrsim b_n$.

In what follows write $C_{\bbeta}$ and $C_{\bbeta,\bd}$ for generic strictly positive constants that can change from line to line and only depend on the model parameters $\bbeta$ and $\bd$ as indicated in the subscript.

\subsection{The Blaschke--Petkantschin formula}

A tool that lies at the heart of many of the computations in the theory of random polytopes is the affine Blaschke--Petkantschin formula, see \cite[Thm. 7.2.7]{SW08}. It is an integral-geometric transformation formula that, in our application, allows to replace $d$-fold integration over $\RR^d$ by first an integration over all $d$-tuples of points that span the same affine hyperplane, followed by an integration over all affine hyperplanes of $\RR^d$. The Jacobian of this transformation has a geometric interpretation as the volume of the convex hull of these points, up to powers and multiplicative constants. 

\begin{proposition}[Blaschke-Petkantschin formula] \label{prop:BP}
Let $f:(\RR^d)^{d}\to\RR$ be a non-negative measurable function. Then,
\begin{align*}
&\int_{(\RR^d)^{d}}f(\bx_1,\ldots,\bx_d)\,\dint(\bx_1,\ldots,\bx_d) \\
    &\quad = {(d-1)!\over 2}\int_{\SS^{d-1}}\int_\RR\int_{H(\bw,s)^{d}}f(\by_1,\ldots,\by_d)
        \vol_{d-1}({\rm co}\, \{\by_1,\ldots,\by_d\})\,\dint_{H(\bw,s)}(\by_1,\ldots,\by_d)\,\dint s\,\sigma_{d-1}(\dint\bw),
\end{align*}
where $\dint_{H(\bw,s)}\bx$ indicates integration with respect to the Lebesgue measure on $H(\bw,s)$.
\end{proposition}

\subsection{Polyspherical coordinates}

To integrate over sets of the form $Z_\bd$ as defined in \eqref{eq:DefZd} we need a kind of spherical coordinate system, which is adapted to their product structure. Given $\bd=(d_1,\dotsc,d_m)\in \mathbb{N}^m$ such that $\sum_{i=1}^m d_i =d$ we can decompose $\RR^d=\prod_{i=1}^m \RR^{d_i}$ into a product of $d_i$-dimensional orthogonal subspaces $\RR^{d_i}$.
Thus, a unit vector $\bw\in\SS^{d-1}$ decomposes uniquely into 
\begin{equation*}
    \bw = \sum_{i=1}^m v_i \bu_i,
\end{equation*}
for $\bu_i \in \SS^{d_i-1}=\RR^{d_i}\cap \SS^d$ and $\bv=(v_1,\dotsc,v_m)\in\SS^{m-1}_+ := \{\bv\in\SS^{m-1}: v_i\geq 0\}$. In other words, $\SS^{d-1}$ can be written as an $M$-product of $\SS^{d_1-1},\dotsc, \SS^{d_m-1}$ for $M=\SS^{m-1}_+$ (or $M=\SS^{m-1}$). In general, by an $M$-product we refer to a set of the form
\begin{equation*}
    \prod_M (K_1,\dotsc,K_m) := \{v_1\bx_1+\dotsc+v_m\bx_m : (v_1,\dotsc,v_m)\in M, \bx_i\in L_i\},
\end{equation*}
where $M\subset \RR^m$ and $K_1,\ldots,K_m\subset \RR^d$ are fixed, see \cite{GHW:2013}.

Integration over $\SS^{d-1}$ can be written as integration over the components of the $M$-product with $M=\SS_+^{m-1}$, a technique which is known under the name of polyspherical coordinates. In what follows, we identify $\SS^{d_i-1}$ with the subsphere $\SS^d\cap [\{0\}^{d_1+\dotsc+d_{i-1}} \times \RR^{d_i} \times \{0\}^{d_{i+1}+\dotsc+d_m}]$, i.e., $\{\SS^{d_i-1}\}_{i\in\{1,\ldots,m\}}$ is a sequence of pairwise orthogonal subspheres on $\SS^{d-1}$.

\begin{proposition}[Polyspherical coordinates]\label{prop:spherical_fubini}
    Let $d\geq 2$ and $f:\SS^{d-1}\to \RR$ be non-negative measureable function. For a sequence of strictly positive integers $\bd=(d_1,\dotsc,d_m)$ such that $\sum_{i=1}^m d_i =d$ we have
    \begin{align}\label{eqn:spherical_fubini}
        \int_{\SS^{d-1}} f(\bw)\, \sigma_{d-1}(\dint\bw) &= 
        \int_{\SS^{m-1}_+} \left[ \int_{\SS^{d_1-1}}\cdots \int_{\SS^{d_m-1}} f\left(\sum_{i=1}^m v_i\bu_i\right) 
            \, \prod_{i=1}^m v_i^{d_i-1} \sigma_{d_i-1}(\dint\bu_i)  \right] \sigma_{m-1}(\dint\bv).
    \end{align}
\end{proposition}
\begin{proof}
    We consider the $0$-homogeneous extension of $f$ to $\RR^d\setminus\{0\}$ defined by $f(\bx) = f(\bx/\|\bx\|_2)$.
    Now, by integration in usual spherical coordinates, writing $\bx=r\bw$ with $r\geq 0$ and $\bw\in\SS^{d-1}$, we have on the one hand
    \begin{equation*}
        \int_{\RR^d} f(\bx)\, e^{-\|\bx\|_2^2} \, \dint\bx 
        = \left(\int_{0}^{\infty} r^{d-1} e^{-r^2}\,\dint r\right) \int_{\SS^{d-1}} f(\bw) \, \sigma_{d-1}(\dint \bw).
    \end{equation*}
    On the other hand, by first decomposing $\RR^d$ in the orthogonal product $\prod_{i=1}^m \RR^{d_i}$ and using spherical coordinates in each of the corresponding subspaces and again on the radii, we find
    \begin{align*}
        &\int_{\RR^d} f(\bx)\, e^{-\|\bx\|_2^2} \, \dint\bx \\
        &\quad= \int_{\RR^{d_i}} \cdots \int_{\RR^{d_m}} f\left(\sum_{i=1}^m \bx_i\right) \prod_{i=1}^m e^{-\|\bx_i\|_2^2} \,\dint\bx_i
            \tag{spherical coordinates $\bx_i = r_i\bu_i$}\\
        &\quad= \int_{[0,\infty)^m} \left[\int_{\SS^{d_1-1}}\cdots \int_{\SS^{d_m-1}} f\left(\sum_{i=1}^m r_i\bu_i\right) 
            \prod_{i=1}^m  e^{-r_i^2}r_i^{d_i-1} \,\sigma_{d_i-1}(\dint\bu_i)\right] \dint\br
                \tag{spherical coordinates $\br=(r_1,\dotsc,r_m)=r\bv$, $\sum d_i = d$}\\
        &\quad= \int_{0}^{\infty} \int_{\SS^{m-1}_+} \Bigg[\int_{\SS^{d_1-1}}\cdots \int_{\SS^{d_m-1}} f\left(r\sum_{i=1}^m v_i\bu_i\right) 
                r^{d-1} e^{-r^2\sum_{i=1}^m v_i^2}
                    \prod_{i=1}^m v_i^{d_i-1}\, \sigma_{d_i-1}(\dint\bu_i)\Bigg] \sigma_{m-1}(\dint\bv) \, \dint r
            \tag{$\sum v_i^2 =1$ and $f(r\bx) = f(\bx)$}\\
        &\quad= \left(\int_{0}^{\infty} r^{d-1} e^{-r^2}\, \dint r\right) \int_{\SS^{m-1}_+} 
            \Bigg[\int_{\SS^{d_1-1}}\cdots \int_{\SS^{d_m-1}} f\left(\sum_{i=1}^m v_i\bu_i\right) 
            \prod_{i=1}^m v_i^{d_i-1} \sigma_{d_i-1}(\dint\bu_i)\Bigg] \sigma_{m-1}(\dint\bv).
    \end{align*}
    This concludes the proof, since $\int_{0}^{\infty} r^{d-1} e^{-r^2}\, \dint r = \frac{1}{2}\Gamma(\frac{d}{2}) <\infty$.
\end{proof}

\begin{remark}
    Note that Proposition \ref{prop:spherical_fubini} can also be deduced from \cite[Lem. 6.5.1]{SW08} by induction on $m\geq 2$.
    Indeed, this follows since for $m=2$ and $d_1=k\in\{1,\dotsc,d-2\}$ the integral formula \eqref{eqn:spherical_fubini} reads as
    \begin{align*}
        &\int_{\SS^{d-1}} f(\bw)\, \sigma_{d-1}(\dint\bw) \\
            &\qquad = \int_{\SS^{k-1}} \int_{\SS^{d-k-1}} 
             \int_{0}^{\frac{\pi}{2}} f\left((\sin \alpha)\bu+(\cos \alpha)\bv\right) 
                \, (\sin\alpha)^{k-1}(\cos\alpha)^{d-k-1}\, \dint\alpha \, \sigma_{d-k-1}(\dint\bv) \, \sigma_{k-1}(\dint\bu).
    \end{align*}
\end{remark}

We note that by symmetry we also have
\begin{equation*}
    \int_{\SS^{d-1}} f(\bw)\, \sigma_{d-1}(\dint\bw) 
    = 2^{-m} \int_{\SS^{m-1}} \left[ \int_{\SS^{k_1-1}}\cdots \int_{\SS^{k_m-1}} f\left(\sum_{i=1}^m v_i\bu_i\right) 
            \, \prod_{i=1}^m |v_i|^{d_i-1} \,\sigma_{d_i-1}(\dint\bu_i)  \right] \sigma_{m-1}(\dint\bv).
\end{equation*}

\subsection{Integral asymptotics}

In the proof of Theorem \ref{thm:main} we need information about the asymptotic behavior of a certain type of integrals. Such integrals have previously been encountered in the theory of random polytopes in \cite[Sec. 3]{AW91}, and in a different but closely related form also in \cite[Lem. 2.4]{ReitznerSchuettWernerPolytopeBoundaryPolytope}. Since we need an extension, we provide an independent proof, which also works for real-valued parameters $\alpha$ and not only for integers as in \cite{AW91}.

\begin{lemma}\label{lm:AffentrangerWieacker}
    Fix $m\in\NN$, $m\geq 1$, and let $\alpha\geq 0$, $a_1\geq a_2\geq\ldots\geq a_m>0$ and $c\in(0,1]$. 
    \begin{itemize}
        \item[(1)] Suppose that $a_{m-1}-a_m>0$. Then
    \begin{equation*}
        \int_{0}^1\ldots\int_{0}^1 (1-cx_1\cdots x_m)^{n-\alpha} x_1^{a_1}\cdots x_m^{a_m} \, \dint x_m\ldots\dint x_1
        = \frac{\Gamma(a_m+1) (cn)^{-(a_m+1)}}{(a_1-a_m)\cdots(a_{m-1}-a_m)} (1+ o_n(1)),
    \end{equation*}

        \item[(2)] Suppose that $\ell\in\{1,\ldots,m-1\}$ is the smallest index such that $a_\ell=\ldots=a_m$. Then
        \begin{equation*}
        \int_{0}^1\ldots\int_{0}^1 (1-cx_1\cdots x_m)^{n-\alpha} x_1^{a_1}\cdots x_m^{a_m} \, \dint x_m\ldots\dint x_1
        = {\Gamma(a_\ell+1)(cn)^{-(a_\ell+1)}(\log n)^{m-\ell}\over(a_1-a_\ell)\cdots(a_{\ell-1}-a_\ell)}(1+o_n(1)).
        \end{equation*}
    \end{itemize}
\end{lemma}
\begin{proof}
If $a_{m-1}-a_m>0$ consider the substitution $x_m(t)=\frac{t}{cnx_1\cdots x_{m-1}}$. This yields
\begin{align*}
I(n) &:= \int_{0}^1\ldots\int_{0}^1 (1-cx_1\cdots x_m)^{n-\alpha} x_1^{a_1}\cdots x_m^{a_m} \, \dint x_m\ldots\dint x_1 \\
&=(cn)^{-(a_m+1)}\int_0^1\ldots\int_0^1\int_0^{cnx_1\cdots x_{m-1}}\Big(1-{t\over n}\Big)^{n-\alpha}t^{a_m}x_1^{a_1-a_m-1}\cdots x_{m-1}^{a_{m-1}-a_m-1}\,\dint t\dint x_{m-1}\ldots\dint x_1.
\end{align*}
Next, we apply Fubini's theorem and integrate with respect to $x_1,\ldots,x_{m-1}$:
\begin{align*}
I(n) &= (cn)^{-(a_m+1)}\int_0^\infty\int_0^1\ldots\int_0^1\int_{t\over cnx_1\cdots x_{m-2}}^1\Big(1-{t\over n}\Big)^{n-\alpha}t^{a_m}x_1^{a_1-a_m-1}\cdots x_{m-1}^{a_{m-1}-a_m-1}\,\dint x_{m-1}\ldots\dint x_1\dint t\\
&={(cn)^{-(a_m+1)}\over a_{m-1}-a_m}\int_0^\infty\int_0^1\ldots\int_0^1\Big(1-{t\over n}\Big)^{n-\alpha}t^{a_m}x_1^{a_1-a_m-1}\cdots x_{m-2}^{a_{m-2}-a_m-1}\\
&\hspace{6cm}\times\Big(1-\Big({t\over cnx_1\cdots x_{m-2}}\Big)^{a_{m-1}-a_m}\Big)\,\dint x_{m-2}\ldots\dint x_1\dint t\\
&={(cn)^{-(a_m+1)}\over (a_1-a_m)\cdots(a_{m-1}-a_m)}(1+o_n(1))\int_0^\infty \Big(1-{t\over n}\Big)^{n-\alpha}t^{a_m}\,\dint t,
\end{align*}
as $n\to\infty$.
This yields the first part of the lemma once we notice that the last integral is $\Gamma(a_m+1)(1+o_n(1))$, as $n\to\infty$.

For the second part we start by writing
\begin{align*}
  J(n) &:=\int_{0}^1\ldots\int_{0}^1 (1-cx_1\cdots x_m)^{n-\alpha} x_1^{a_1}\cdots x_m^{a_m} \, \dint x_m\ldots\dint x_1\\
  &= \int_0^1\ldots\int_0^1 x_1^{a_1}\cdots x_{\ell-1}^{a_{\ell-1}}\,J_1(n;x_1,\ldots,x_{\ell-1})\,\dint x_{\ell-1}\cdots\dint x_1  ,
\end{align*}
with
\begin{align*}
    J_1(n;x_1,\ldots,x_{\ell-1}) := \int_0^1\ldots\int_0^1 (1-cx_1\cdots x_m)^{n-\alpha}(x_\ell\cdots x_m)^{a_\ell}\,\dint x_\ell\ldots\dint x_m.
\end{align*}
To deal with the asymptotic behavior of $J_1(n;x_1,\ldots,x_{\ell-1})$, as $n\to\infty$, we follow the first steps of the proof of \cite[Lem.\ 3.1]{AW91}, which deals with the case that $a_\ell$ and $\alpha$ are integers (however, these steps to not use this property). This yields that
\begin{align*}
    J_1(n;x_1,\ldots,x_{\ell-1})  
    &=(1+o_n(1)){(\log(cx_1\cdots x_{\ell-1}(n-\alpha)))^{m-\ell}\over (cx_1\cdots x_{\ell-1}(n-\alpha))^{a_\ell+1}}
        J_2(n;x_1,\ldots,x_{\ell-1}),
\end{align*}
with
\begin{align*}
   J_2(n;x_1,\ldots,x_{\ell-1}):=\int_0^{cx_1\cdots x_{\ell-1}(n-\alpha)}\Big(1-{z\over n-\alpha}\Big)^{n-\alpha}z^{a_\ell}\,\dint z. 
\end{align*}
To deal with the asymptotics of $J_2(n;x_1,\ldots,x_{\ell-1})$, as $n\to\infty$, we deviate from the argument of \cite{AW91} and instead use the substitution $z=cx_1\cdots x_{\ell-1}(n-\alpha)y$ to first rewrite $J_2(n;x_1,\ldots,x_{\ell-1})$ in the form
\begin{equation*}
    J_2(n;x_1,\ldots,x_{\ell-1}) = (cx_1\cdots x_{\ell-1}(n-\alpha))^{a_\ell+1}\int_0^1 e^{-nh(y)}\varphi(y)\,\dint y,
\end{equation*}
with
\begin{equation*}
    h(y) := -\log(1-cx_1\cdots x_{\ell-1}y)\qquad\text{and}\qquad\varphi(y) := {y^{a_\ell}\over(1-cx_1\cdots x_{\ell-1}y)^\alpha}.
\end{equation*}
The next step is to observe that the function $h$ takes its minimum zero at $y=0$ and that
\begin{align*}
    h(y) &= cx_1\cdots x_{\ell-1}y + {(cx_1\cdots x_{\ell-1}y)^2\over 2} + O(y^3),\\
    \varphi(y) &= x^{a_\ell} + c\alpha x_1\cdots x_{\ell-1} y^{a_\ell+1} + O(y^{a_\ell+2}),
\end{align*}
as $y\to 0$. This brings us into the position to apply \cite[Thm. II.1.1] {WongBook} with $\mu=1$, $a_0=cx_1\cdots x_{\ell-1}$, $\alpha=a_\ell+1$, $b_0=1$ and $c_0={b_0\over\mu a_0^{\alpha/\mu}}=(cx_1\cdots x_{\ell-1})^{-(a_\ell+1)}$ in the notation of \cite{WongBook}. This yields
\begin{align*}
    J_2(n;x_1\,\ldots, x_{\ell-1}) 
        &=(1+o_n(1))\Gamma(a_\ell+1){(n-\alpha)^{a_\ell+1}\over n^{a_\ell+1}},
\end{align*}
as $n\to\infty$. In fact, computing the second-order term in this expansion shows that the sequence $o_n(1)$ can be chosen independently of $x_1,\ldots,x_{\ell-1}$. Plugging this back into $J_1(n;x_1,\ldots, x_{\ell-1})$ we arrive at
\begin{align*}
    J_1(n;x_1,\ldots, x_{\ell-1}) 
        &=(1+o_n(1)) {(\log(cx_1\cdots x_{\ell-1}(n-\alpha)))^{m-\ell}\over (cx_1\cdots x_{\ell-1}(n-\alpha))^{a_\ell+1}}
            \Gamma(a_\ell+1){(n-\alpha)^{a_\ell+1}\over n^{a_\ell+1}}\\
        &=(1+o_n(1))\Gamma(a_\ell+1)(cx_1\cdots x_{\ell-1}n)^{-(a_\ell+1)}(\log(cx_1\cdots x_{\ell-1}(n-\alpha)))^{m-\ell}\\
        &=(1+o_n(1))\Gamma(a_\ell+1)(cx_1\cdots x_{\ell-1}n)^{-(a_\ell+1)}(\log n)^{m-\ell}.
\end{align*}
Finally, this can be combined with the representation of $J(n)$ to conclude that
\begin{align*}
    J(n) &= (1+o_n(1)) \Gamma(a_\ell+1)(cn)^{-(a_\ell+1)}(\log n)^{m-\ell}
        \int_0^1\ldots\int_0^1 x_1^{a_1-a_\ell-1}\ldots x_{\ell-1}^{a_{\ell-1}-a_\ell-1}\,\dint x_{\ell-1}\ldots\dint x_1\\
    &=(1+o_n(1)){\Gamma(a_\ell+1)(cn)^{-(a_\ell+1)}(\log n)^{m-\ell}\over(a_1-a_\ell)\ldots(a_{\ell-1}-a_\ell)},
\end{align*}
as claimed.
\end{proof}

\section{Reduction to the meta-cube and geometric estimates}\label{sec:GeometricIngredients}

In this section, we introduce the concept of meta-objects, which allow us to describe the geometry of our $d$-dimensional product body $Z_{\bd}$ in terms of a simple geometric object: the $m$-dimensional cube $B_{\infty}^m:=[-1,1]^m$. Throughout this section and in the remainder of the document, we will use the following notation and conventions. Recall from \eqref{eqn:beta_density} the definition of the beta density $f_{\beta,d}$ on $\RR^d$ with parameter $\beta>-1$. The normalization constant of this density will be denoted by
\begin{equation}\label{eq:CbetaConstant}
c_{\beta,d}:=\frac{\Gamma(\frac{d_j}{2}+\beta_j+1)}{\pi^{\frac{d}{2}}\Gamma(\beta_j+1)},\qquad\qquad\text{ and we put }
\qquad\qquad c_{\beta}:=c_{\beta,1},\qquad c_{\beta,0}:=1,
\end{equation}
Moreover, if $\bbeta=(\beta_1,\ldots,\beta_m)$ and $\bd=(d_1,\ldots,d_m)$ for some $m\geq 1$, $\beta_1,\ldots,\beta_m>-1$ and $d_1,\ldots,d_m\geq 1$, we set
\[
f_{\bbeta,\bd}(\bx) := f_{\bbeta,\bd}((\by_1,\ldots,\by_m)) = \prod_{i=1}^m f_{\beta_i,d_i}(\by_i),
\]
where $\bx$ is a point in $\RR^{d}$, $d:=d_1+\ldots+d_m$, with block coordinates $\by_1\in\RR^{d_1},\ldots,\by_m\in\RR^{d_m}$. The $\bbeta$-content of a Borel set $A\subseteq\RR^d$ will be denoted by
\[
\PP_{d}(A;\bbeta,\bd):=\int_{A}f_{\bbeta,\bd}(\bx)\,\dint \bx.
\]
In the notation we use, the lower index $d$ will always be the sum of the coordinates of the vector $\bd$. In particular, if $\bd=\bone:=(1,\ldots,1)$, we put
$\PP_{m}(A;\bbeta):=\PP_{m}(A;\bbeta,\bone)$ for Borel sets $A\subseteq\RR^m$.
Moreover, for an affine hyperplane $H\subset\RR^d$ we define
\[
\vol_{d-1}(A\cap H;\bbeta,\bd):=\int_{A\cap H}f_{\bbeta, \bd}(\bx)\,\dint_H \bx,
\]
with $\dint_H\bx$ referring to the integration with respect to the Lebesgue measure on $H$, and use the abbreviation $\vol_{d-1}(A\cap H;\bbeta)$ for $\vol_{d-1}(A\cap H;\bbeta,\bone)$. 

It will turn out to be convenient for us to extend the definition of $\vol_{d-1}(A\cap H;\bbeta)$ to the case $d=1$ in the following way. A zero-dimensional affine subspace $H(\bv,s)$ with normal vector $\bv=(\pm 1)\in\SS^0=\{-1,+1\}$ and distance $s\in\RR$ from the origin is identified with the point $s\bv = (\pm s)\in\RR$. Also note that $\RR^0$ is isometric to the single point at the origin. Thus, by formally putting $\bx:=\by+s\bv=(\pm s)$ and $\by:=(0)\in\RR^0$, we can write
\begin{equation}\label{eq:ZeroVolumeDef}
    \vol_{0}(A\cap H(\bv,s);\bbeta)
    = \int_{A\cap H(\bv,s)} c_{\beta}(1-|s\bv|^2-\|\by\|_2^2)^{\beta}\,\dint_{H(\bv,s)}\by
    := c_{\beta}(1-s^2)^{\beta}{\bf 1}\{s\in A\},
\end{equation}
where $A\subseteq\RR$ is a Borel set and $\beta>-1$.

\subsection{Meta-objects}\label{sec:MetaObjects}

\begin{figure}
    \centering
    \noindent
    \begin{tikzpicture}
        \begin{scope}
        \clip (-3.25,-3.25) rectangle (3.25,3.25);
        \node at (0,0) {
            \includegraphics[width=8cm]{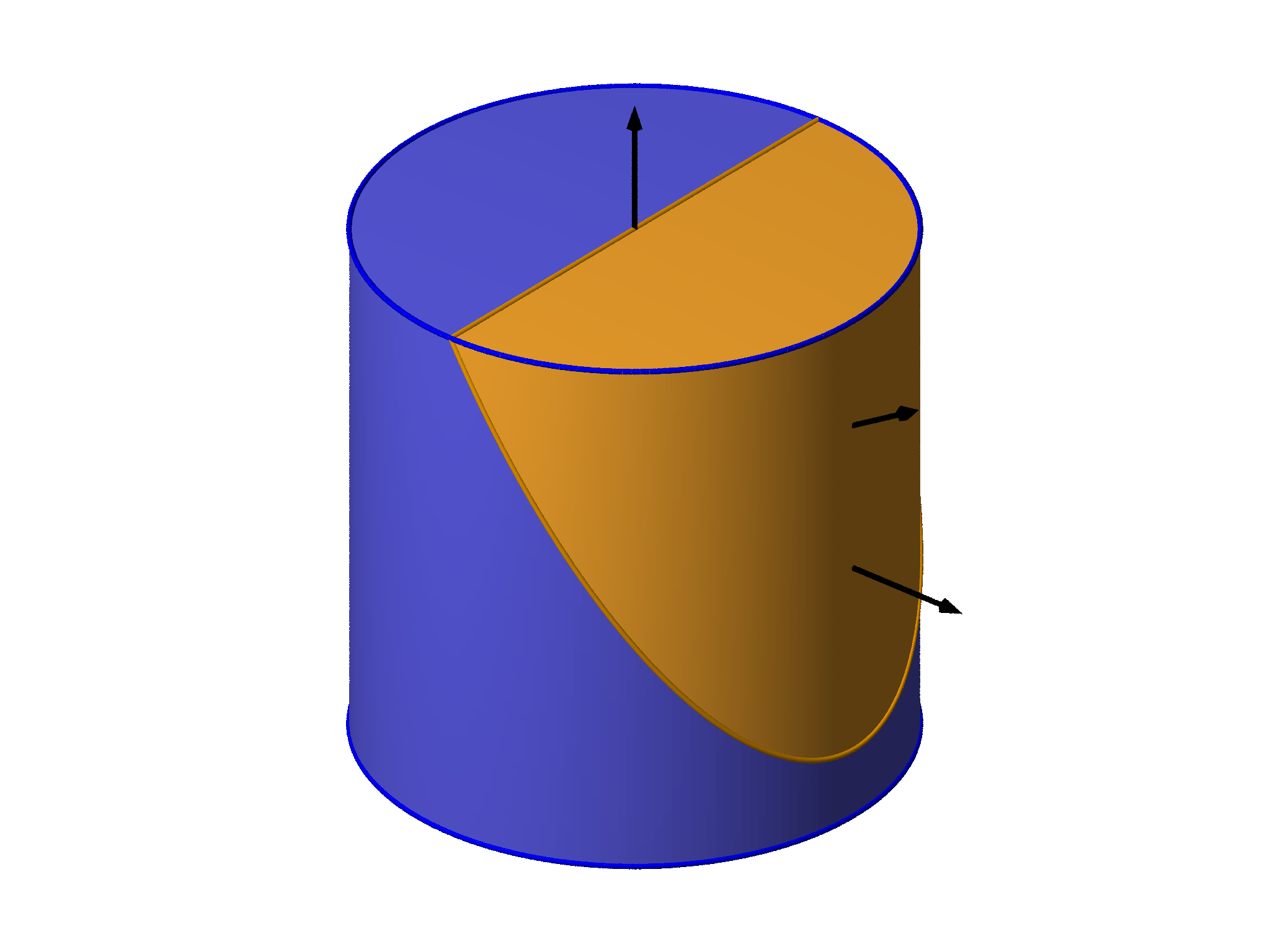}};
        \end{scope}
        \node at (-20:2.5) {$\bu_1$};
        \node at (3.3,0.6) {$\bw=v_1\bu_1+v_2\bu_2$};
        \node at (-0.25,2) {$\bu_2$};

        \node at (35:3.8) {$Z_{(2,1)}\cap H^+(\bw,s)$};
        \useasboundingbox (-2.5,-3.25) rectangle (5,3.25);
        
    \end{tikzpicture}
    \hfill
    \begin{tikzpicture}[scale=2.]
        \draw[left color=blue!80!white, right color=blue!80!white, middle color=blue!30!white] (-1,-1) rectangle (1,1);
        \draw[left color=orange!30!white, right color=orange] (0,1) -- (1,1) -- (1,{1-sqrt(3)})--cycle;

        \draw[dashed] (0,0) -- ({-1.5*cos(60)},{1.5*sin(60)});
        \draw[dashed] (0,1) -- ({-0.5*cos(60)},{0.5*sin(60)+1});
        \draw ({-1.25*cos(60)},{1.25*sin(60)}) --++ ({0.5*cos(30)},{0.5*sin(30)})
            node[midway, above] {$s$};

        \draw[->] (0,0) -- (1.5,0) node[below] {$\be_1$};
        \draw[->] (0,0) -- (0,1.5) node[left] {$\be_2$};
        \draw[->] (0,0) -- (30:1.5) node[right] {$\bv=(v_1,v_2)$};

        \node at (0.58,0.75) {$C^+(\bv,s)$};
        \node[below left] at (0,0) {$0$};

        \useasboundingbox (-1.25,-1.5) rectangle (2.5,1.75);
    \end{tikzpicture}
    \caption{In the left figure we see a cap $Z_{(2,1)}\cap H^+(\bw,s)$ (in orange) of the cylinder $Z_{(2,1)}\subset \RR^3$ (in blue) and on the right we see the corresponding meta-cap $C^+(\bv,s)\subset [-1,1]^2$.}
    \label{fig:cylinder}

\end{figure}
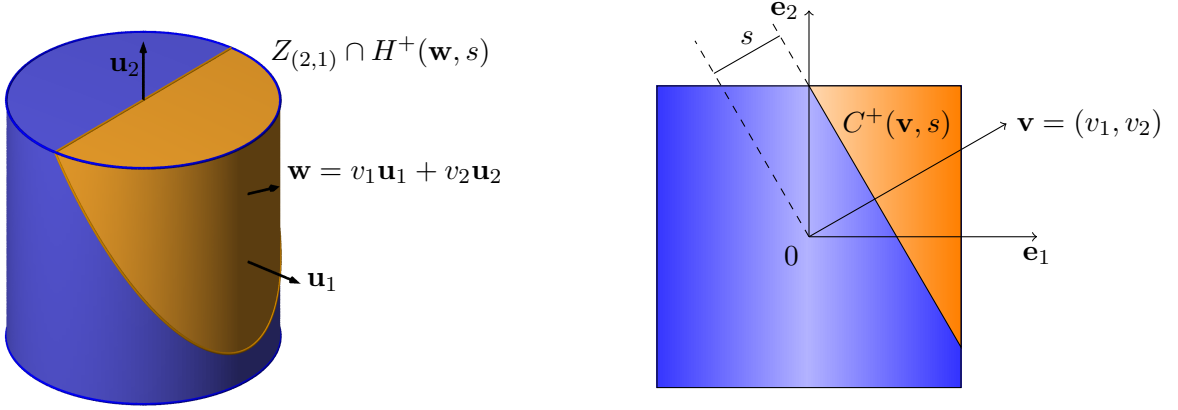

Let $m\geq 1$, $d_1,\ldots,d_m\geq 1$ and $\bd=(d_1,\ldots,d_m)$. The product body $Z_{\bd}$, defined at \eqref{eq:DefZd}, is symmetric with respect to the direct product $\mathrm{SO}(\bd):=\prod_{i=1}^m \mathrm{SO}(d_i)$ of the special orthogonal groups $\mathrm{SO}(d_1),\ldots,\mathrm{SO}(d_m)$. The orbit of a point $\bx=(\bx^{(1)},\dotsc,\bx^{(m)})\in Z_{\bd}$ with respect to $\mathrm{SO}(\bd)$ is denoted by
\begin{equation*}
    [\bx]_{\mathrm{SO}(\bd)} = \{(\widetilde\bx^{(1)},\dotsc,\widetilde\bx^{(m)}) : \text{such that $\|\widetilde\bx^{(i)}\|_2=\|\bx^{(i)}\|_2$ for all $i\in\{1,\ldots,m\}$}\}.
\end{equation*}
In particular, we have that $W_{m}:=Z_{\bd} / \mathrm{SO}(\bd) \cong [0,1]^m$, i.e., a vector $\br=(r_1,\dotsc,r_m) \in [0,1]^m$ uniquely determines an orbit
\begin{equation*}
    \br\cong [(r_1\bu_1,\dotsc,r_m\bu_m)]_{\mathrm{SO}(\bd)}\in W_{m},
\end{equation*}
where $\bu_i\in \SS^{d_i-1}$ is chosen arbitrarily for all $i\in\{1,\ldots,m\}$. At the same time, for any $\bx\in Z_{\bd}$ there is a unique $\bz\in W_m$ and $\bu_i\in \SS^{d_i-1}$, $i\in\{1,\ldots,m\}$, such that $\bx=(z_1\bu_1,\ldots, z_m\bu_m)$. Further, for any choice of vectors $\bu_i\in \SS^{d_i-1}$, $i\in\{1,\ldots,m\}$ consider the $m$-dimensional linear subspace
\[
L(\bu_1,\ldots,\bu_m):={\rm lin}\{(0,\ldots,0, \bu_i, 0,\ldots,0)\colon i\in\{1,\ldots,m\}\},
\]
where ${\rm lin}$ stands for the linear hull (span), and denote by ${\rm proj}_{\bu_1,\ldots,\bu_m}:\RR^d\to L(\bu_1,\ldots,\bu_m)$ the orthogonal projection operator onto $L(\bu_1,\ldots,\bu_m)$. Due to the $\mathrm{SO}(\bd)$-symmetry of $Z_{\bd}$ we thus have
\[
{\rm proj}_{\bu_1,\ldots,\bu_m}(Z_{\bd})=\prod_{i=1}^m \mathrm{co}\{-\bu_i,\bu_i\} \cong B^m_{\infty}. 
\]
We refer to $B^m_{\infty}$ as the \emph{meta-cube} associated with $Z_{\bd}$. 

For our purposes, it will turn out to be enough to consider the geometry of the meta-cube $B^m_{\infty}$ in order to understand the  geometry of the underlying product body $Z_{\bd}$. In particular, if we consider a hyperplane $H(\bw,s)=\{\bx\in \RR^d: \bx\cdot \bw=s\}$ with normal direction $\bw\in\SS^{d-1}$ and (signed) distance $s\in \RR$ to the origin, we find that the intersections $Z_{\bd}\cap H(\bw,s)$ and the caps $Z_{\bd}\cap H^+(\bw,s)$ are congruent to $Z_{\bd}\cap H(\widetilde \bw,s)$ and $Z_{\bd}\cap H^+(\widetilde \bw,s)$, respectively, for all $\widetilde \bw\in[\bw]_{\mathrm{SO}(\bd)}$. In other words, for any $\bv=(v_1,\ldots,v_m)\in\SS^{m-1}$ all the caps $Z_{\bd}\cap H^+(\bw,s)$ with $\bw = (v_1 \bu_1,\ldots,v_m\bu_m)$ are congruent. Also note that
\[
{\rm proj}_{\bu_1,\ldots,\bu_m}(Z_{\bd}\cap H^+(\bw,s))\cong B^m_{\infty}\cap H^+(\bv,s).
\]
We therefore refer to the section 
\begin{equation*}
    C(\bv,s):=B^m_{\infty}\cap H(\bv,s)\subset \RR^m,
\end{equation*}
and to the cap
\begin{equation*}
    C^+(\bv,s):=B^m_{\infty}\cap H^+(\bv,s)\subset \RR^m,
\end{equation*}
of the meta-cube as the \emph{meta-section} and \emph{meta-cap} 
with \emph{meta-normal direction} $\bv\in\SS^{m-1}$, respectively. Up to $\mathrm{SO}(\bd)$-transformations, the meta-section $C(\bv,s)$ uniquely determines the shape of all sections $Z_{\bd}\cap H(\bw,s)\subset \RR^d$, while the meta-cap $C^+(\bv,s)$ uniquely determines the shape of all caps $Z_{\bd}\cap H^+(\bw,s)\subset \RR^d$ for $\bw = (v_1 \bu_1,\ldots,v_m\bu_m)$, where $\bu_i\in\SS^{d_i-1}$ is arbitrary for all $i\in\{1,\ldots,m\}$, see Figure \ref{fig:cylinder} for an illustration in $\RR^3$.

\medskip

The next lemma provides the crucial connection between the $\beta$-content of a meta-cap and a meta-section and a cap and a section corresponding to the product body, respectively.

\begin{lemma}\label{lem:ReductionLemma}
    For any $\bw\in\SS^{d-1}$ let $\bv=(v_1,\ldots,v_m)\in \SS^{m-1}_+\subset W_m$ be such that $\bw = (v_1 \bu_1,\ldots,v_m\bu_m)$ for some $\bu_i\in\SS^{d_i-1}$, $i\in\{1,\ldots,m\}$. Then, for all $\bd=(d_1,\ldots,d_m)$ with $d_i\ge 1$ and all $\bbeta=(\beta_1,\ldots,\beta_m)$ with $\beta_i>-1$, $i\in\{1,\ldots,m\}$ we have that
        \begin{equation*}
        \PP_{d}(Z_{\bd}\cap H^+(\bw,s);\bbeta,\bd) = \left(\prod_{i=1}^m \frac{c_{\beta_i,d_i}}{c_{\beta_i,d_i-1}c_{\beta_i}}\right) \PP_{m}(C^+(\bv,s);(\bd-\bone)/2+\bbeta),
    \end{equation*}
    and
    \begin{equation}\label{eq:BetaArea}
    \vol_{d-1}(Z_{\bd}\cap H(\bw,s); \bbeta,\bd)=\left(\prod_{i=1}^m \frac{c_{\beta_i,d_i}}{c_{\beta_i,d_i-1}c_{\beta_i}}\right)\vol_{m-1}(C(\bv,s);(\bd-\bone)/2+\bbeta),
    \end{equation}
    where $c_{\beta,d}$ and $c_{\beta}$ are the normalizing constants of the beta density given by \eqref{eq:CbetaConstant}.
\end{lemma}

\begin{proof}
Due to $\mathrm{SO}(\bd)$-symmetry of $Z_{\bd}$ and $\mathrm{SO}(\bd)$-invariance of $f_{\bbeta, \bd}(\bx)$, and since the cap $Z_{\bd}\cap H^+(\bw,s)$ and the section $Z_{\bd}\cap H(\bw,s)$ are congruent to the cap $Z_{\bd}\cap H^+(\widetilde\bw,s)$ and the section $Z_{\bd}\cap H(\widetilde\bw,s)$ for any $\widetilde\bw\in[\bw]_{\mathrm{SO}(\bd)}$, respectively, without loss of generality we may assume that
\[
\bu_i=(1,0,\ldots,0),\qquad i\in\{1,\ldots,m\}.
\]
By definition of a meta-cap we have
\begin{align*}
     \PP_{d}(Z_{\bd}\cap H^+(\bw,s);\bbeta,\bd)&=\int_{Z_{\bd}\cap H^+(\bw,s)}f_{\bbeta, \bd}(\bx)\,\dint \bx\\
     &=\int_{B^{d_1}_2}\ldots\int_{B^{d_m}_2}{\bf 1}\{(\by_1,\ldots,\by_m)\in H^+(\bw,s)\}\prod_{i=1}^mc_{\beta_i,d_i}(1-\|\by_i\|_2^2)^{\beta_i}\,\dint \by_i.
     \end{align*}
 We use the following change of variables. If $d_i=1$ we write $z_i=\by_i\in B^1_2=[-1,1]$ and, otherwise, $z_i=y_{i,1}\in[-1,1]$ and $\widetilde\by_i=(y_{i,2},\ldots,y_{i,d_i})$, satisfying $\|\widetilde\by_i\|_2\leq \sqrt{1-z_i^2}$ for $i\in\{1,\ldots,m\}$. Further, recall that $\bu_i=(1,0,\ldots,0)$ and, hence, the condition $(\by_1,\ldots,\by_m)\in H^+(\bw,s)$ is equivalent to 
\[
\sum_{i=1}^mv_i (\by_i\cdot \bu_i)\ge s\qquad  \Longleftrightarrow\qquad \sum_{i=1}^mv_iz_i\ge s,
\]
with our new choice of coordinates. The latter implies $\bz\in H^+(\bv,s)$ and together with $\bz\in B_{\infty}^m$ we arrive at 
     \begin{align*}
     \PP_{d}(Z_{\bd}\cap H^+(\bw,s);\bbeta,\bd)&=\int_{B_{\infty}^m}{\bf 1}\{(z_1,\ldots,z_m)\in H^+(\bv,s)\}\\
&\qquad\qquad\times\prod_{i=1}^m\int_{\RR^{d_i-1}}c_{\beta_i,d_i}(1-z_i^2-\|\widetilde\by_i\|_2^2)^{\beta_i}{\bf 1}\big\{\|\widetilde\by_i\|_2\leq \sqrt{1-z_i^2}\big\}\,\dint \widetilde\by_i\,\dint\bz\\
&=\int_{C^+(\bv,s)}\prod_{i=1}^m \frac{c_{\beta_i,d_i}}{c_{\beta_i,d_i-1}} (1-z_i^2)^{\frac{d_i-1}{2}+\beta_i}\,\dint\bz\\
&=\left(\prod_{i=1}^m \frac{c_{\beta_i,d_i}}{c_{\beta_i,d_i-1}c_{\beta_i}}\right)\PP_{m}(C^+(\bv,s);(\bd-\bone)/2+\bbeta),
\end{align*}
where we recall our convention that for $d_i=1$ we write $\int_{\RR^0}(1-z^2-\|\by\|_2^2)^{\beta}\,\dint \by=(1-z^2)^{\beta}$ and $c_{\beta,0}=1$.

To show \eqref{eq:BetaArea} we proceed in the same manner and start by writing
\[
\vol_{d-1}(Z_{\bd}\cap H(\bw,s); \bbeta,\bd)=\int_{Z_{\bd}\cap H(\bw,s)}\prod_{i=1}^mc_{\beta_i,d_i}(1-\|\by_i\|_2^2)^{\beta_i}\,\dint_{H(\bw,s)} (\by_i).
\]
Using the same change of variables $z_i=y_{i,1}\in[-1,1]$ and $\widetilde\by_i=(y_{i,2},\ldots,y_{i,d_i})$ if $d_i\neq 1$ for $i\in\{1,\ldots,m\}$ as before, we obtain
\begin{align*}
\vol_{d-1}&(Z_{\bd}\cap H(\bw,s); \bbeta,\bd)\\
&=\int_{C(\bv,s)}\prod_{i=1}^m\int_{\RR^{d_i-1}}c_{\beta_i,d_i}(1-z_i^2-\|\widetilde\by_i\|_2^2)^{\beta_i}{\bf 1}\big\{\|\widetilde\by_i\|_2\leq \sqrt{1-z_i^2}\big\}\,\dint \widetilde\by_i\,\dint_{H(\bv,s)}\bz\\
&=\left(\prod_{i=1}^m \frac{c_{\beta_i,d_i}}{c_{\beta_i,d_i-1}c_{\beta_i}}\right)\vol_{m-1}(C(\bv,s);(\bd-\bone)/2+\bbeta),
\end{align*}
where we recall convention \eqref{eq:ZeroVolumeDef} for the case $m=1$.
    \end{proof}

\subsection{Geometry of meta-caps and meta-sections}

In this subsection we obtain some bounds for the $\beta$-content of the meta-caps $C^+(\bv,s)$ and the meta-sections $C(\bv,s)$. They will turn out to be one of the crucial geometric estimates we require in oder to prove Theorem \ref{thm:main}. Let $\bv\in\SS^{m-1}_+$ and set 
    \begin{equation*}
        \|\bv\|_1 := \sum_{i=1}^m v_i \quad \text{and}\quad  s_1(\bv) := \|\bv\|_1- 2\min_{i\in\{1,\ldots,m\}} v_i.
    \end{equation*}

Note that for $m=1$ we have $\SS^0_+=\{1\}$ and therefore $\|\bv\|_1 = 1$ and $s_1(\bv)=-1$ for all $\bv\in\SS^0_+$.

\begin{lemma}\label{lem:AreaVolumeCapEstimates}
    For any $\bv\in\SS^{m-1}_+$, $\bbeta\in (-1,\infty)^m$ and $-\|\bv\|_1<s<\|\bv\|_1$ we have
    \begin{equation*}
        \PP_{m}(C^+(\bv,s);\bbeta) 
            \asymp \prod_{i=1}^{m} \left(\min\left\{\frac{\|\bv\|_1-s}{2v_i},1\right\}\right)^{\beta_i+1}.
    \end{equation*}
    Furthermore, if $m=1$, $s\ge 0$ and $\beta\in(-1,\infty)$, or if $m\geq 2$ and $\bbeta\in [0,\infty)^{m}$, then
    \begin{equation*}
        \vol_{m-1}(C(\bv,s);\bbeta) 
            \lesssim\,(\|\bv\|_1-s)^{-1}\, \prod_{i=1}^{m} \left(\min\left\{\frac{\|\bv\|_1-s}{2v_i},1\right\}\right)^{\beta_i+1}.
    \end{equation*}
\end{lemma}

\begin{proof}
   We start by considering $\PP_{m}(C^+(\bv,s);\bbeta)$. By definition we have 
   \begin{align}
    \PP_{m}(C^+(\bv,s);\bbeta) 
        &= \prod_{i=1}^m c_{\beta_i}\int_{C^+(\bv,s)}\prod_{i=1}^m(1-y_i^2)^{\beta_i}\,\dint y_i\notag\\
        &=\prod_{i=1}^m c_{\beta_i}\int_{B_{\infty}^m\cap H^+(\bv,s)}
            \prod_{i=1}^m(1-y_i)^{\beta_i}(1+y_i)^{\beta_i}\,\dint y_i.\label{eq:Pbeta3}
   \end{align}
   Let $\bz_0=\bone=(1,\ldots,1)$ be the vertex of the meta-cube $B_{\infty}^m$, which, as it is easy to ensure, belongs to the cap $B_{\infty}^m\cap H^+(\bv,s)$ for all $s<\|\bv\|_1$, since $\bz_0\cdot \bv=\|\bv\|_1$. Further, let 
   \begin{equation*}
    t_i:=\max\{q\ge 0\colon \bone-q\,\be_i\in B_{\infty}^m\cap H^+(\bv,s)\},\qquad i\in\{1,\ldots,m\},
   \end{equation*}
   where $\be_1,\ldots, \be_m$ are unit vectors of the standard orthonormal basis of $\RR^m$, and set $\bz_i:=\bz_0-t_i\be_i$. We note that
   \begin{equation*}
        t_i=\min\left\{\frac{\|\bv\|_1-s}{v_i},2\right\}.
   \end{equation*}
   In fact, this identity follows from the fact that $\bz_0-q\,\be_i\in H^+(\bv,s)$ is equivalent to
   \begin{equation*}
        \bz_0\cdot \bv -q\,v_i\ge s\qquad \Longleftrightarrow\qquad q\leq {\|\bv\|_1-s\over v_i},
   \end{equation*}
   while $\bz_0-q\,\be_i\in B_{\infty}^m$ implies $q\leq 2$. Also we note that by construction we have
   \begin{equation*}
        \mathrm{co}\, \{\bz_0,\bz_1,\ldots,\bz_m\}\subset B_{\infty}^m\cap H^+(\bv,s),
   \end{equation*}
   and by \eqref{eq:Pbeta3} we get
   \begin{equation}\label{eq:Pbeta5}
    \PP_{m}(C^+(\bv,s);\bbeta) 
        \ge \prod_{i=1}^m c_{\beta_i} \int_{\mathrm{co}\, \{\bz_0,\bz_1,\ldots,\bz_m\}}\prod_{i=1}^m(1-y_i^2)^{\beta_i}\,\dint y_i.
   \end{equation}
   
    Consider the simplex 
    \begin{equation*}
    \mathrm{co}\,\{\bz_0,\bz_1,\ldots,\bz_m\}
        =\left\{\sum_{i=0}^m\alpha_i\bz_i\colon \sum_{i=0}^m\alpha_i=1, \alpha_i\in [0,1]\right\}.
    \end{equation*}
    For any $\by\in \mathrm{co}\,\{\bz_0,\bz_1,\ldots,\bz_m\}$ we thus have $y_i=1-\alpha_it_i$, $i\in\{1,\ldots,m\}$. We choose  new coordinates $(z_1,\dotsc,z_m)\in[0,1]^m$ introduced by
    \begin{equation}\label{eq:change_of_variables}
        y_i = 1 - t_i (1-z_i) \prod_{\ell=1}^{i-1} z_\ell,\qquad i\in\{1,\ldots,m\},
    \end{equation}
    where the empty product in case $i=1$ is interpreted to be equal to $1$. Note that for any $(z_1,\dotsc,z_m)\in[0,1]^m$ we have
    \begin{equation*}
        \sum_{i=1}^m (1-z_i) \prod_{\ell=1}^{i-1} z_\ell
            =\sum_{i=1}^m\Big(\prod_{\ell=1}^{i-1} z_\ell-\prod_{\ell=1}^{i} z_\ell\Big)
            =1-\prod_{i=1}^mz_{\ell}\in[0,1],
    \end{equation*}
    and, hence, $\by\in \mathrm{co}\,\{\by_0,\by_1,\ldots,\by_m\}$ by choosing $\alpha_0=\prod_{i=1}^mz_{\ell}\in[0,1]$ and $\alpha_i=(1-z_i) \prod_{\ell=1}^{i-1} z_\ell\in[0,1]$ for $i\in\{1,\ldots,m\}$. Moreover, the Jacobian of the transformation \eqref{eq:change_of_variables} is
    \begin{equation*}
        \left|\det\left(\left[\frac{\dint y_{i}}{\dint z_\ell}\right]_{i,\ell=1}^m\right)\right| = \prod_{i=1}^m t_iz_i^{m-i}.
    \end{equation*}
   
   Combining this with \eqref{eq:Pbeta5} we obtain
   \begin{align}
    \PP_{m}(C^+(\bv,s);\bbeta) 
        &\ge \left(\,\prod_{i=1}^m2^{\beta_i}c_{\beta_i}\left(\min\left\{\frac{\|\bv\|_1-s}{v_i},2\right\}\right)^{\beta_i+1}\right)
            \int_{[0,1]^m}\widetilde g(\bz;s,\bv) \,\dint\bz,\label{eq:Pbeta4}
   \end{align}
   where the function
   \begin{equation*}
    \widetilde g(\bz;s,\bv)
        :=\prod_{i=1}^m z_i^{m-i}(1-z_i)^{\beta_i} \left(\prod_{\ell=1}^{i-1} z_\ell^{\beta_i}\right)
                 \left(1-\min\left\{\frac{\|\bv\|_1-s}{2v_i},1\right\}(1-z_i)\prod_{\ell=1}^{i-1} z_\ell\right)^{\beta_i},
   \end{equation*}
   is positive on $[0,1]^m$.
   Further, we note that for any $\bz\in[1/4;3/4]^m$ we have 
   \begin{equation*}
    \widetilde g(\bz;s,\bv)
        \ge 2^{-C_{\bbeta}}\prod_{i=1}^m \min\left\{\left(1-{3\over 4}\min\left\{\frac{\|\bv\|_1-s}{2v_i},1\right\}\right)^{\beta_i} , 1\right\}
        \ge 2^{-C_{\bbeta}}>0.
   \end{equation*}
   Hence, the integral in \eqref{eq:Pbeta4} is bounded below as follows:
   \begin{equation*}
    \int_{[0,1]^m}\widetilde g(\bz;s,\bv) \dint\bz
        \ge \int_{[1/4;3/4]^m}\widetilde g(\bz;s,\bv) \dint\bz
        \ge C_{\bbeta}>0.
   \end{equation*}
   Together with \eqref{eq:Pbeta4} this finishes the proof of the lower bound in the first claim.
   
   For the upper bound we note that for any $\by\in C^+(\bv,s)=B_{\infty}^m\cap H^+(\bv,s)$ we have
   \begin{equation*}
        1-y_i\leq \min\left\{\frac{\|\bv\|_1-s}{v_i},2\right\},\qquad i\in\{1,\ldots,m\}.
   \end{equation*}
   Indeed, since $\by\in B_{\infty}^m$ we have $1-y_i\leq 2$ for all $i\in\{1,\ldots,m\}$. On the other hand, if $\by\in H^+(\bv,s)$, then $\by\cdot\bv\ge s$. Hence,
   \begin{equation*}
        1-y_i
            \leq {v_i+\sum_{j=1,j\neq i}^m y_j v_j - s \over v_i} 
            \leq \frac{\|\bv\|_1-s}{v_i}.
   \end{equation*}
   Thus, by \eqref{eq:Pbeta3} and by substituting $z_i=(1-y_i)/2$, we see that
   \begin{align*}
    \PP_{m}(C^+(\bv,s);\bbeta) 
        &\leq \prod_{i=1}^m c_{\beta_i} \int_{\RR} (1-y_i)^{\beta_i} (1+y_i)^{\beta_i} \, 
                {\bf 1}\left\{1-\min\left\{\frac{\|\bv\|_1-s}{v_i},2\right\}\leq y_i\leq 1\right\}\dint y_i \\
        &= \prod_{i=1}^m 2^{2\beta_i+1} c_{\beta_i} \int_{0}^{\min\{\frac{\|\bv\|_1-s}{2v_i},1\}} z_i^{\beta_i} (1-z_i)^{\beta_i}\, \dint\bz_i\\
        &= C_{\bbeta} \prod_{i=1}^m B\left(1+\beta_i,1+\beta_i; \min\left\{\frac{\|\bv\|_1-s}{2v_i},1\right\}\right),
   \end{align*}
   where for $a,b>0$, $B(a,b;x) = \int_{0}^x z^{a-1} (1-z)^{b-1} \, \dint z$, $x\in[0,1]$, is the incomplete beta function.
   For $t,x\in[0,1]$ we have $1-t\leq 1-tx\leq 1$, which yields
   \begin{equation*}
       B(a,b;x) = x^a \int_{0}^1 t^{a-1} (1-xt)^{b-1}\, \dint t \leq
       \begin{cases}
            x^a/a &\text{if $b\geq 1$,}\\
            x^a B(a,b) &\text{if $b\in(0,1)$,}
       \end{cases}
   \end{equation*}
   for all $x\in[0,1]$.
   Thus,
   \begin{equation*}
       \PP_{m}(C^+(\bv,s);\bbeta) \leq C_{\bbeta} \prod_{i=1}^m \Big(\min\left\{\frac{\|\bv\|_1-s}{v_i},2\right\}\Big)^{\beta_i+1}.
   \end{equation*}
   This complete the first part of the proof.

    \smallskip
   
   Now consider $ \vol_{m-1}(C(\bv,s);\bbeta)$. We start by noting that for $m=1$ we have $\bv=(1)$ and $\|\bv\|_1=1$. Then, by \eqref{eq:ZeroVolumeDef}, for any $0\leq s<1$ we get
   \begin{align*}
   \vol_0(C(\bv,s);\bbeta)&=c_{\beta}(1-s^2)^{\beta}\leq \max\{2^{\beta},1\} \, c_{\beta}(1-s)^{\beta}.
   \end{align*}
   Moreover, in this case we also have
   \begin{equation*}
        (\|\bv\|_1-s)^{-1}\left(\min\left\{{\|\bv\|_1-s\over v_1},2\right\}\right)^{\beta+1}=(1-s)^{\beta},
   \end{equation*}
   and the claim for $m=1$ follows. 
   
   So, let $m\ge 2$. Then, by definition, we have
   \begin{align}
    \vol_{m-1}(C(\bv,s);\bbeta)
        &=\prod_{i=1}^m c_{\beta_i} 
            \int_{B_{\infty}^m\cap H(\bv,s)}\prod_{i=1}^m(1-y_i)^{\beta_i}(1+y_i)^{\beta_i}\,\dint_{H(\bv,s)} \by .\label{eq:BetaArea3}
   \end{align}
   Since $C(\bv,s)\subset C^+(\bv,s)$ we conclude that for any $\by\in B_{\infty}^m\cap H(\bv,s)$ it holds that
   \begin{equation*}
        0\leq 1-y_i\leq \min\left\{\frac{\|\bv\|_1-s}{v_i},2\right\},\qquad i\in\{1,\ldots,m\},
   \end{equation*}
   as follows by the previous argument.
   
   In order to estimate $\vol_{m-1}(C(\bv,s);\bbeta)$ we proceed as follows.
   Assume without loss of generality that $v_m\geq v_i$ for all $i\in\{1,\dotsc,m-1\}$. Then $1\leq \|\bv\|_1 \leq m v_m$ yields $1\geq v_m \geq 1/\sqrt{m}$ and 
   \begin{equation*}
       1-y_m\leq \frac{\|\bv\|_1-s}{v_m} \leq \frac{\|\bv\|_1-s}{v_i} \qquad \text{for all $i\in\{1,\dotsc,m-1\}$}.
   \end{equation*}
   Let $H_{m}:=\{\bx\in\RR^m\colon x_m=0\}\cong \RR^{m-1}$ and consider the orthogonal projection $\widetilde C(\bv,s)$ of $C(\bv,s)$ onto $H_{m}$. Note, that $\be_m$ is a normal vector of $H_{m}$ and $C(\bv,s)\subset H(\bv,s)$ with normal vector $\bv$. 
   Hence we can parametrize $C(\bv,s)$ by $y_m(y_1,\dotsc,y_{m-1}) = (s-\sum_{i=1}^{m-1} v_i y_i)/v_m$, where the relative Jacobian is 
   \begin{equation*}
       \sqrt{1+\|\nabla y_m\|_2^2} = \sqrt{1+\sum_{i=1}^{m-1} \frac{v_i^2}{v_m^2}} = \frac{\|\bv\|_2}{v_m}=\frac{1}{v_m}.
   \end{equation*}
   For $\beta_m\geq 0$ we have 
   \begin{align*}
        (1-y_m^2)^{\beta_m} 
        \leq 2^{\beta_m} \left(\frac{\|\bv\|_1-s}{v_m}\right)^{\beta_m}.
   \end{align*}
   Thus, from \eqref{eq:BetaArea3} we derive
   \begin{align*}
        \vol_{m-1}(C(\bv,s);\bbeta) 
            &\leq \frac{C_{\bbeta}}{v_m} \left(\frac{\|\bv\|_1-s}{v_m}\right)^{\beta_m} \int_{\tilde{C}(\bv,s)} 
                    \prod_{i=1}^{m-1} (1-y_i)^{\beta_i} \, \dint y_i\\
            &= \frac{C_{\bbeta}}{\|\bv\|_1-s} \left(\frac{\|\bv\|_1-s}{v_m}\right)^{\beta_m+1}
                \int_{\tilde{C}(\bv,s)}
                \prod_{i=1}^{m-1} (1-y_i)^{\beta_i} \, \dint y_i.
   \end{align*}

   From the previous argument for any $\widetilde \by\in \widetilde C(\bv,s)$ we have
   \begin{equation*}
        1-\widetilde{y}_i = 1- y_i\leq \min\left\{\frac{\|\bv\|_1-s}{v_i},2\right\}\qquad \text{for all $i\in\{1,\ldots,m-1\}$},
   \end{equation*}
   since $\widetilde \by=(y_1,\ldots,y_{m-1})$ for some $\by=(y_1,\ldots,y_m)\in C(\bv,s)$. 
   This leads to
   \begin{align*}
        &\vol_{m-1}(\widetilde C(\bv,s);\bbeta)\\
            &\leq \frac{C_{\bbeta}}{\|\bv\|_1-s} \left(\frac{\|\bv\|_1-s}{v_m}\right)^{\beta_m+1}
                \prod_{i=1}^{m-1}\int_{\RR} (1-y_i^2)^{\beta_i} 
                    {\bf 1}\left\{1-\min\left\{\frac{\|\bv\|_1-s}{v_i},2\right\}\leq y_i\leq 1\right\}\dint y_i\\
            &\leq \frac{C_{\bbeta}}{\|\bv\|_1-s} \left(\frac{\|\bv\|_1-s}{v_m}\right)^{\beta_m+1} \prod_{i=1}^{m-1} B\left(1+\beta_i,1+\beta_i; \min\left\{\frac{\|\bv\|_1-s}{2v_i},1\right\}\right)\\
            &\leq \frac{C_{\bbeta}}{\|\bv\|_1-s}  
           \left(\frac{\|\bv\|_1-s}{v_m}\right)^{\beta_m+1}
                \prod_{i=1}^{m-1} \left(\min\left\{\frac{\|\bv\|_1-s}{2v_i},1\right\}\right)^{\beta_i+1}.
   \end{align*}
   Finally, we note that 
   for $s> -\|\bv\|_1$ and since $v_m\geq 1/\sqrt{m}$, we get
   \begin{equation*}
        {\|\bv\|_1-s\over 2v_m}< \sqrt{m}\, \|\bv\|_1\leq m
            = m\min\left\{{\|\bv\|_1-s\over 2v_m},1\right\},
   \end{equation*}
   which finishes the proof.
\end{proof}

\begin{lemma}\label{lem:CapVolume}
For any $\bv\in\SS^{m-1}_+$, $\bbeta\in (-1,\infty)^m$ and $s\in(s_1(\bv),\|\bv\|_1)$ we have
    \begin{align*}
        \PP_{m}(C^+(\bv,s);\bbeta) 
             &\asymp(\|\bv\|_1-s)^{\beta+m} \prod_{i=1}^m v_i^{-\beta_i-1}.
    \end{align*}
    Note that $s\in (s_1(\bv),\|\bv\|_1)$ implies in particular that the latter interval is not empty and $\|\bv\|_1\neq s_1(\bv)$, which is equivalent to $v_i>0$ for all $i\in\{1,\ldots,m\}$.
    \end{lemma}

    \begin{proof}
    This bound is a direct corollary of Lemma \ref{lem:AreaVolumeCapEstimates}, since for $s_1(\bv)<s<\|\bv\|_1$ and any $i\in\{1,\ldots,m\}$ we have
    \begin{equation*}
    {\|\bv\|_1-s\over v_i}\leq {\|\bv\|_1-s_1(\bv)\over v_i}\leq 2. \qedhere
    \end{equation*}
    \end{proof}
    
\begin{lemma}\label{lm:CapArea}
        For any $\bv\in\SS^{m-1}_+$, $\bbeta\in [0,\infty)^m$ and $s\in(\max\{s_1(\bv),0\},\|\bv\|_1)$ we have
        \begin{align*}
             \vol_{m-1}(C(\bv,s);\bbeta)&\asymp (\|\bv\|_1-s)^{\beta+m-1}  \prod_{i=1}^m v_i^{-\beta_i-1}.
        \end{align*}
        For $m=1$ we may choose $\beta_1>-1$.
        Note also that for $m\ge 2$ we have $s_1(\bv)\ge 0$, while for $m=1$ we have $s_1(\bv)=-1$.
\end{lemma}

\begin{proof}
The upper bound is direct consequence of Lemma \ref{lem:AreaVolumeCapEstimates}, since for $s_1(\bv)<s<\|\bv\|_1$ and any $i\in\{1,\ldots,m\}$ we have
    \[
    {\|\bv\|_1-s\over v_i}\leq {\|\bv\|_1-s_1(\bv)\over v_i}\leq 2.
    \]

    For the lower bound we again distinguish the cases $m=1$ and $m\ge 2$. For $m=1$ we have $\bv=(1)$, $\|\bv\|_1=1$ and $s_1(\bv)=-1$. Further, note that by \eqref{eq:ZeroVolumeDef} for any $0<s<1$ and $\beta\ge 0$ we have
   \begin{align*}
        \vol_0(C(\bv,s);\bbeta)&=c_{\beta}(1-s^2)^{\beta}\ge \min\{2^{\beta},1\} c_{\beta}(1-s)^{\beta},
   \end{align*}
   which together with
   \[
   (\|\bv\|_1-s)^{-1}\Big(\min\Big\{{\|\bv\|_1-s\over v_1},2\Big\}\Big)^{\beta+1}=(1-s)^{\beta},
   \]
   finishes the proof in this case. 

   \medskip
   If $m\ge 2$ we use the representation
    \begin{align}
    \vol_{m-1}(C(\bv,s);\bbeta)&=\prod_{i=1}^m c_{\beta_i}\int_{C(\bv,s)}\prod_{i=1}^m(1-y_i^2)^{\beta_i}\,\dint_{H(\bv,s)} \by.\label{eq:AreaBeta1}
    \end{align}

    Further, set 
    \begin{equation}\label{eq:EpsilonDefinition1}
    \varepsilon_i := {\|\bv\|_1-s\over 2mv_i}\in (0,1),\qquad i\in\{1,\ldots,m\},\qquad\text{ for }s\in (s_1(\bv),\|\bv\|_1),
    \end{equation}
    and consider for $\boldsymbol{\varepsilon}:=(\varepsilon_1,\ldots,\varepsilon_m)$ and $c\in [0,1]$ the embedded box $B_{\infty}^m(c\,\boldsymbol{\varepsilon})=\prod_{i=1}^m[-1+c\varepsilon_i,1-c\varepsilon_i]\subset B_{\infty}^m$. First, let us point out that  for any $c\in[0,1]$ we have
    \[
    \emptyset\neq B_{\infty}^m(c\,\boldsymbol{\varepsilon})\cap H(\bv,s)\subset B_{\infty}^m\cap H(\bv,s)=C(\bv,s).
    \]
    Indeed, consider the vertex $\widetilde\by_0=(1-c\varepsilon_1,\ldots,1-c\varepsilon_m)$ of $B_{\infty}^m(c\,\boldsymbol{\varepsilon})$. Since $s<\|v\|_1$ we have
    \[
    \bv\cdot \widetilde\by_0=\sum_{i=1}^mv_i\Big(1-{c(\|\bv\|_1-s)\over 2mv_i}\Big)=\|\bv\|_1-{c(\|\bv\|_1-s)\over 2}=\|\bv\|_1\Big(1-{c\over 2}\Big)+{cs\over 2}>s,
    \]
    which implies $\by_0\in B_{\infty}^m(c\,\boldsymbol{\varepsilon})\cap H^+(\bv,s)\neq \emptyset$. At the same time, one can ensure that $\widetilde\by_0$ is the only vertex of  $B_{\infty}^m(c\,\boldsymbol{\varepsilon})$ contained in $B_{\infty}^m(c\,\boldsymbol{\varepsilon})\cap H^+(\bv,s)$. More precisely, without loss of generality assume that $v_m=\min_{i\in\{1,\ldots,m\}}v_i$. Then for any vertex $\widetilde {\boldsymbol{\delta}}$ of $B_{\infty}^m(c\,\boldsymbol{\varepsilon})$ by taking into account that $\|\bv\|_1\ge s_1(\bv)$ and $s>s_1(\bv)$ we have
\begin{align*}
\widetilde {\boldsymbol{\delta}}\cdot\bv&\leq \sum_{i=1}^{m-1}\Big(v_i-{c(\|\bv\|_1-s)\over 2m}\Big)-v_m+{c(\|\bv\|_1-s)\over 2m}\\
&=s_1(\bv)-{c(m-2)\over 2m}\|\bv\|_1+{c(m-2)\over 2m}s\\
&= {c(m-2)\over 2m}(s_1(\bv)-\|\bv\|_1)+{c(m-2)\over 2m}s+{m(2-c)+2c\over 2m}s_1(\bv)< s.
\end{align*}
    Hence, for any $c\in[0,1]$ the set $B_{\infty}^m(c\,\boldsymbol{\varepsilon})\cap H^+(\bv,s)$ is an $m$-dimensional simplex with vertices $\widetilde \by_0$ and
    \[
    \widetilde \by_i=\widetilde\by_0-\widetilde t_i\be_i\in B_{\infty}^m(c\,\boldsymbol{\varepsilon})\cap H(\bv,s),\qquad i\in\{1,\ldots,m\},
    \]
    where $\be_1,\ldots,\be_m$ are again the unit vectors of the standard orthonormal basis of $\RR^m$. In order to determine $\widetilde t_i$ we note that $\widetilde \by_i\in H(\bv,s)$ implies 
    \[
    s=\widetilde \by_i\cdot\bv=\widetilde \by_0\cdot \bv - \widetilde t_i v_i=\|\bv\|_1\Big(1-{c\over 2}\Big)+{cs\over 2}- \widetilde t_i v_i\qquad \Longleftrightarrow\qquad \widetilde t_i=\Big(1-{c\over 2}\Big)\Big({\|\bv\|_1-s\over v_i}\Big).
    \]
    Thus,
    \begin{equation*}
        \vol_m(B_{\infty}^m(c\,\varepsilon) \cap H^+(\bv,s)) = \frac{1}{m!}\prod_{i=1}^m \tilde{t}_i = {\big(1-{c\over 2}\big)^m\over m!}(\|\bv\|_1-s)^m \prod_{i=1}^m v_i^{-1}.
    \end{equation*}
    In particular, we obtain that
    \begin{align}
        \nonumber\vol_{m-1}(B_{\infty}^m(c\,\boldsymbol{\varepsilon})\cap H(\bv,s)) &=-{\dint\over \dint s}\vol_m(B_{\infty}^m(c\,\boldsymbol{\varepsilon})\cap H^+(\bv,s))\\
        &={\big(1-{c\over 2}\big)^m\over (m-1)!}(\|\bv\|_1-s)^{m-1}\prod_{i=1}^mv_i^{-1}.\label{eq:Volume_m-1}
    \end{align}
    Further, we note that $B_{\infty}^m(\boldsymbol{\varepsilon})\subset B_{\infty}^m((1/2)\,\boldsymbol{\varepsilon})$ and for any $\by\in  B_{\infty}^m((1/2)\,\boldsymbol{\varepsilon})\setminus B_{\infty}^m(\boldsymbol{\varepsilon})$ we have that $1-\varepsilon_i\leq y_i\leq 1-\varepsilon_i/2$ for all $i\in\{1,\ldots,m\}$.
    Now, combining the above observation together with Equalities \eqref{eq:Volume_m-1} and \eqref{eq:AreaBeta1} for any $\bbeta\in (-1,\infty)^m$ we get 
    \begin{align*}
    \vol_{m-1}(C(\bv,s);\bbeta)&\ge \prod_{i=1}^mc_{\beta_i} \int_{(B_{\infty}^m((1/2)\,\boldsymbol{\varepsilon})\setminus (B_{\infty}^m(\boldsymbol{\varepsilon}))\cap H(\bv,s)}\prod_{i=1}^m(1-y_i^2)^{\beta_i}\dint_{H(\bv,s)} \by\notag\\
    &\ge \prod_{i=1}^m c_{\beta_i}\min\big\{(1-(1-\varepsilon_i)^2)^{\beta_i},(1-(1-\varepsilon_i/2)^2)^{\beta_i}\big\}\\
    &\qquad\qquad\times\Big(\vol_{m-1}(B_{\infty}^m((1/2)\,\boldsymbol{\varepsilon})\cap H(\bv,s))-\vol_{m-1}(B_{\infty}^m(\boldsymbol{\varepsilon})\cap H(\bv,s))\Big)\notag\\
    &\ge C_{\bbeta}(\|\bv\|_1-s)^{\beta+m-1}\prod_{i=1}^m v_i^{-\beta_i-1},
    \end{align*}
    where we additionally used \eqref{eq:EpsilonDefinition1} and that $2\ge 2-\varepsilon_i/2\ge 2-\varepsilon_i\ge 1$ for all $i\in\{1,\ldots,m\}$. This finishes the proof.
\end{proof}

\subsection{Bounds for a Sylvester-type functional}\label{sec:Sylvester}

Let $\PP_{H(\bw,s)}(\,\cdot\,;\bbeta,\bd)$ be the probability measure appearing as a restriction of $\PP_{d}(\,\cdot\,;\bbeta,\bd)$ to the hyperplane $H(\bw,s)$. Formally, for any Borel set $A\subset H(\bw,s)\cap Z_{\bd}$ we have
\[
 \PP_{H(\bw,s)}(A;\bbeta,\bd)=\int_{A} \widetilde f_{\bbeta,\bd}(\bx)\, \dint_{H(\bw,s)}\bx,
\]
where
\[
\widetilde f_{\bbeta, \bd}(\bx):=\frac{f_{\bd,\bbeta}(\bx)}{\vol_{d-1}(Z_{\bd}\cap H(\bw,s);\bbeta,\bd)},
\]
is a normalized version of the block-beta density $f_{\bbeta,\bd}$. Let $\tilde{\bX}_1,\dotsc,\tilde{\bX}_d$ be independent random points with distribution $ \PP_{H(\bw,s)}(\,\cdot\,;\bbeta,\bd)$ and consider the \textit{Sylvester-type functional} $G(\bw,s)$, defined by
\begin{equation*}\label{eq:SylvesterFunctional}
    G(\bw,s): = \EE\left[ \frac{\vol_{d-1}(\mathrm{co}\{\tilde{\bX}_1,\dotsc,\tilde{\bX}_d\})}{\vol_{d-1}(Z_{\bd}\cap H(\bw,s))}\right].
\end{equation*}
A functional of this type is well-known to appear in the analysis of random polytopes, see, for example, \cite{Reitzner2005}. It does therefore come at no surprise, that it also shows up in our considerations, especially in the proof of Theorem \ref{thm:main}.

Note that since $\widetilde f_{\bbeta,\bd}$ and $\vol_{d-1}$ are $\mathrm{SO}(\bd)$-invariant we have $G(\bw,s)=G(\widetilde \bw,s)$ for all $\widetilde \bw\in[\bw]_{\mathrm{SO}(\bd)}$. Thus, $G(\bw,s)$ only depends on the meta-vector $\bv$ of $\bw$ and on $s$. In view of this, we will sometimes abuse notation and write $G(\bv,s)$ instead of $G(\bw,s)$. 

In the next lemma we show that $G(\bv,s)$ is bounded from above and below uniformly for all $\bv\in\SS_+^{m-1}$ and $s\in(s_1(\bv),\|\bv\|_1)$. Although much more could be said about $G(\bw,s)$, this is the only property we require about this functional in our further analysis.

\begin{lemma}\label{lem:SylvesterFunctional}
  For any $\bbeta\in[0,\infty)^m$, any $\bw=(v_1\bu_1,\ldots,v_m\bu_m)$ with $\bu_i\in\SS^{d_i-1}$ and $\bv\in\SS^{m}_+$, and for any $s\in(\max\{0,s_1(\bv)\},\|\bv\|_1)$ we have
    \begin{equation*}
       0< C_{\bd,\bbeta}\leq  G(\bw,s) \leq 1.
    \end{equation*}
\end{lemma}

\begin{proof}
    The upper bound is a trivial consequence of the fact that 
    \[
    \frac{\vol_{d-1}(\mathrm{co}\{\bx_1,\dotsc,\bx_d\})}{\vol_{d-1}(Z_{\bd}\cap H(\bw,s))}\leq 1,
    \]
    for all $\bx_1,\ldots,\bx_d\in Z_{\bd}\cap H(\bw,s)$.

    For the lower bound we set 

\begin{equation}\label{eq:EpsilonDefinition}
    \varepsilon_i := {\|\bv\|_1-s\over 2mv_i}>0,\qquad i\in\{1,\ldots,m\},\qquad\text{ for }s\in (s_1(\bv),\|\bv\|_1).
    \end{equation}
    We note that for any $i\in\{1,\ldots,m\}$, since $s\in(s_1(\bv),\|\bv\|_1)$, we have
    \begin{equation}\label{eq:EpsilonBound}
        \varepsilon_i = {\|\bv\|_1-s\over 2mv_i}< {\|\bv\|_1-s_1(\bv)\over 2mv_i}=\frac{\min_{j\in\{1,\ldots,m\}}v_j}{mv_i}\leq {1\over m}\leq 1.
    \end{equation}
    For $\boldsymbol{\varepsilon}:=(\varepsilon_1,\ldots,\varepsilon_m)$ consider the set 
    \[
    Z_{\bd}(\boldsymbol{\varepsilon})=\{(\bx^{(1)},\ldots,\bx^{(m)})\colon \|\bx^{(i)}\|_2\leq 1-\varepsilon_i,\, i\in\{1,\ldots,m\}\}.
    \]
    With this definition it follows that $((1-\varepsilon_1)\bu_1,\ldots,(1-\varepsilon_m)\bu_m)\in Z_{\bd}(\boldsymbol{\varepsilon})\cap H^+(\bw,s)$ and, hence, $Z_{\bd}(\boldsymbol{\varepsilon})\cap H(\bw,s)\neq \emptyset$. Moreover, $Z_{\bd}(\boldsymbol{\varepsilon})\cap H(\bw,s)\subset Z_{\bd}\cap H(\bw,s)$. Further, for any $\bx=(\bx^{(1)},\ldots,\bx^{(m)})\in Z_{\bd}(\boldsymbol{\varepsilon})\cap H(\bw,s)$ we have
    \[
    f_{\bbeta,\bd}(\bx)=\prod_{i=1}^mc_{\beta_i,d_i}(1-\|\bx^{(i)}\|_2^2)^{\beta_i}\ge C_{\bbeta}\prod_{i=1}^m\varepsilon_i^{\beta_i}(2-\varepsilon_i)^{\beta_i}\ge C_{\bbeta}(\|\bv\|_1-s)^{\beta}\prod_{i=1}^mv_i^{-\beta_i},
    \]
    as follows from \eqref{eq:EpsilonDefinition} and \eqref{eq:EpsilonBound}. Combining Lemma \ref{lem:ReductionLemma} and Lemma \ref{lm:CapArea} we obtain
    \begin{align*}
        \vol_{d-1}(Z_{\bd}\cap H(\bw,s);\bbeta,\bd) &\leq C_{\bbeta}(\|\bv\|_1-s)^{{d+m\over 2}+\beta-1}\prod_{i=1}^mv_i^{-{d_i+1\over 2}-\beta_i},
        \intertext{and}
        \vol_{d-1}(Z_{\bd}\cap H(\bw,s)) &\leq C_{\bbeta}(\|\bv\|_1-s)^{{d+m\over 2}-1}\prod_{i=1}^mv_i^{-{d_i+1\over 2}}.
    \end{align*}
    Combining these estimates we get
    \begin{align}
        G(\bw,s)&=\int_{(Z_{\bd}\cap H(\bw,s))^d}\frac{\vol_{d-1}(\mathrm{co}\{\bx_1,\dotsc,\bx_d\})}{\vol_{d-1}(Z_{\bd}\cap H(\bw,s))}\prod_{i=1}^d\widetilde f_{\bbeta,d}(\bx_i)\dint_{H(\bw,s)} \bx_i\notag\\
        &\ge C_{\bd,\bbeta}(\|\bv\|_1-s)^{-(d+1)({d+m\over 2}-1)}\prod_{i=1}^mv_i^{(d+1)(d_i+1)\over 2}\nonumber\\
        &\hspace{4cm}\times\int_{(Z_{\bd}(\boldsymbol{\varepsilon})\cap H(\bw,s))^d}\vol_{d-1}(\mathrm{co}\{\bx_1,\dotsc,\bx_d\})\prod_{i=1}^d\dint_{H(\bw,s)}\bx_i\notag\\
        &=C_{\bd,\bbeta}(\|\bv\|_1-s)^{-(d+1)({d+m\over 2}-1)}\prod_{i=1}^mv_i^{(d+1)(d_i+1)\over 2}\big(\vol_{d-1}(Z_{\bd}(\boldsymbol{\varepsilon})\cap H(\bw,s))\big)^{d+1}\notag\\
        &\hspace{8cm}\times\EE\left[\frac{\vol_{d-1}(\mathrm{co}\{\tilde{\bY}_1,\dotsc,\tilde{\bY}_d\})}{\vol_{d-1}(Z_{\bd}(\boldsymbol{\varepsilon})\cap H(\bw,s))}\right],\label{eq:Gwsestimate}
    \end{align}
    where $\tilde{\bY}_1,\dotsc,\tilde{\bY}_d$ are independent random points distributed uniformly in $Z_{\bd}(\boldsymbol{\varepsilon})\cap H(\bw,s)$. 
    
    Now, we note that from the main theorem of \cite{Groemer74} (with its predecessor \cite{BlaschkeVierpunkt} dealing with the planar case $d=2$) it follows that 
    \begin{equation}\label{eq:SylvesterGroemer}
    \EE\left[\frac{\vol_{d-1}(\mathrm{co}\{\tilde{\bY}_1,\dotsc,\tilde{\bY}_d\})}{\vol_{d-1}(Z_{\bd}(\boldsymbol{\varepsilon})\cap H(\bw,s))}\right]\ge \EE\left[\frac{\vol_{d-1}(\mathrm{co}\{\tilde{\bZ}_1,\dotsc,\tilde{\bZ}_d\})}{\vol_{d-1}(B_{2}^{d-1})}\right]=:C_{\bd}>0,
    \end{equation}
    where $\tilde{\bZ}_1,\dotsc,\tilde{\bZ}_d$ are independent random points uniformly distributed in the $(d-1)$-dimensional unit ball $B_{2}^{d-1}$. The constant $C_{\bd}$ can explicitly be computed as in \cite[Theorem 8.2.3]{SW08} and yields the corresponding lower bound for the expectation in \eqref{eq:Gwsestimate}.

    It remains to estimate $\vol_{d-1}(Z_{\bd}(\boldsymbol{\varepsilon})\cap H(\bw,s))$ from below. In what follows we will show, that
    \begin{align}\label{eq:18-09-24A}
    \vol_{d-1}(Z_{\bd}(\boldsymbol{\varepsilon})\cap H(\bw,s))\ge C_{\bd}(\|\bv\|_1-s)^{{d+m\over 2}-1}\prod_{i=1}^mv_i^{-{d_i+1\over 2}},
    \end{align}
    which together with \eqref{eq:Gwsestimate} and \eqref{eq:SylvesterGroemer} will finish the proof.
    
    We start by considering the case when $\bd=(d_1)=(d)$, corresponding to $m=1$. In this case we have $\bv=(v_1)=(1)$, $s_1(\bv)=-1$ and 
    \[
    Z_{\bd}(\boldsymbol{\varepsilon})=\Big\{\bx\in\RR^d\colon \|\bx\|_2\leq 1-\varepsilon_1={1+s\over 2}\Big\}.
    \]
    The latter is the $d$-dimensional ball of radius ${1+s\over 2}\in (s,1)$, where $0<s<1$. Moreover, $Z_{\bd}(\boldsymbol{\varepsilon})\cap H(\bw,s)$ is a $(d-1)$-dimensional ball of radius ${1\over 2}\sqrt{(1-s)(1+3s)}\ge {1\over 2}\sqrt{(1-s)}$. This directly implies
    \[
     \vol_{d-1}(Z_{\bd}(\boldsymbol{\varepsilon})\cap H(\bw,s))\ge C_{\bd}(1-s)^{d-1\over 2}=C_{\bd}(\|\bv\|_1-s)^{d-1\over 2}v_1^{-{d+1\over 2}},
    \]
    and the proof of \eqref{eq:18-09-24A} in this case follows. If on the other hand $m\ge 2$, we first note that 
    \[
    \proj_{\bu_1,\ldots,\bu_m}(Z_{\bd}(\boldsymbol{\varepsilon}))\cong B_{\infty}^m(\boldsymbol{\varepsilon})=\prod_{i=1}^m[-1+\varepsilon_i,1-\varepsilon_i],
    \]
    where we recall that $\bw=(v_1\bu_1,\ldots,v_m\bu_m)$, $\bu_i\in\SS^{d_i-1}$, $i\in\{1,\ldots,m\}$ and $\proj_{\bu_1,\ldots,\bu_m}$ is defined in Section \ref{sec:MetaObjects}. Further, using the same arguments as in the proof of Lemma \ref{lem:ReductionLemma} we introduce the change of variables $z_i=y_{i,1}\in [-1+\varepsilon_i,1-\varepsilon_i]$ and $\widetilde \by_i=(y_{i,2},\ldots,y_{i,d_i})$ if $d_i\neq 1$ for $i\in\{1,\ldots,m\}$. Then, we get
    \begin{align}\notag
        \vol_{d-1}(Z_{\bd}(\boldsymbol{\varepsilon})\cap H(\bw,s))&=\int_{Z_{\bd}(\boldsymbol{\varepsilon})\cap H(\bw,s)}\dint_{H(\bw,s)}(\by_1,\ldots,\by_m)\\
        &=C_{\bd}\int_{B_{\infty}^m(\boldsymbol{\varepsilon})\cap H(\bv,s)}
            \left(\prod_{i=1}^m((1-\varepsilon_i)^2-z_i^2)^{d_i-1\over 2}\right) \,\dint_{H(\bv,s)}\bz. \label{eqn:foobar1}
    \end{align}
    We continue noe as in the proof of the lower bound in Lemma \ref{lm:CapArea}. More precisely, we define $\widetilde\varepsilon_i:=3\varepsilon_i/2$, $i\in\{1,\ldots,m\}$, and consider for $\widetilde{\boldsymbol{\varepsilon}}:=(\widetilde{\varepsilon_1},\ldots,\widetilde{\varepsilon_m})$ the box $B_{\infty}^m(\widetilde{\boldsymbol{\varepsilon}})=\prod_{i=1}^m[-1+\widetilde\varepsilon_i,1-\widetilde\varepsilon_i]\subset B_{\infty}^m(\boldsymbol{\varepsilon})$. Again, it is easy to check, that a vertex $\widetilde\by_0=(1-\widetilde\varepsilon_1,\ldots,1-\widetilde\varepsilon_m)$ belongs to $B_{\infty}^m(\widetilde{\boldsymbol{\varepsilon}})\cap C^+(\bv,s)$, since
    \[
    \bv\cdot\widetilde\by_0={\|\bv\|_1\over 4}+{3s\over 4}>s,
    \]
    for any $s\in (s_1(\bv),\|\bv\|_1)$ and $B_{\infty}^m(\widetilde{\boldsymbol{\varepsilon}})\cap C(\bv,s)\neq \emptyset$. Moreover, $\widetilde\by_0$ is the only vertex of $B_{\infty}^m(\widetilde{\boldsymbol{\varepsilon}})$ contained in $B_{\infty}^m(\widetilde{\boldsymbol{\varepsilon}})\cap C^+(\bv,s)$, implying that the latter is an $m$-dimensional  simplex with vertices $\tilde{\by_0}$ and 
    \[
    \widetilde\by_i=\widetilde\by_0-\Big({\|\bv\|_1-s\over 4v_i}\Big)\be_i,\qquad i\in\{1,\ldots,m\},
    \]
    where $\be_1,\ldots,\be_m$ are the unit vectors of the standard orthonormal basis of $\RR^m$. This implies that
    \begin{equation}\label{eqn:foobar2}
    \vol_{m-1}(B_{\infty}^m(\widetilde{\boldsymbol{\varepsilon}})\cap H(\bv,s))=-{\dint\over \dint s}\vol_m(B_{\infty}^m(\widetilde{\boldsymbol{\varepsilon}})\cap H^+(\bv,s))=C_{m}(\|\bv\|_1-s)^{m-1}\prod_{i=1}^mv_i^{-1}.
    \end{equation}
    For $z_i\in[-1+\tilde{\varepsilon}_i, 1-\tilde{\varepsilon}_i]$ we have
    \begin{equation}\label{eqn:foobar3}
        (1-\varepsilon_i)^2-z_i^2 \geq (1-\varepsilon_i)^2-(1-\tilde{\varepsilon}_i)^2 
        = \frac{\varepsilon_i}{2} \left(2-\frac{5}{2}\varepsilon_i\right).
    \end{equation}
    Combing \eqref{eqn:foobar1}, \eqref{eqn:foobar2}, and \eqref{eqn:foobar3}, we find that
    \begin{align*}
        \vol_{d-1}(Z_{\bd}(\boldsymbol{\varepsilon})\cap H(\bw,s))
            &\ge C_{\bd}\int_{B_{\infty}^m(\widetilde{\boldsymbol{\varepsilon}})\cap H(\bv,s)}
                \left(\prod_{i=1}^m((1-\varepsilon_i)^2-z_i^2)^{d_i-1\over 2}\right) \dint_{H(\bv,s)}\bz\\
            &\ge C_{\bd} \left(\prod_{i=1}^m\varepsilon_i^{d_i-1\over 2}\Big(2-{5\over 2}\varepsilon_i\Big)^{d_i-1\over 2} \right)
                \vol_{m-1}(B_{\infty}^m(\widetilde{\boldsymbol{\varepsilon}})\cap H(\bv,s))\\
    &\ge C_{\bd}\Big(2-{5\over 2m}\Big)^{d-m\over 2}(\|\bv\|_1-s)^{{d+m\over 2}-1}\prod_{i=1}^mv_i^{-{d_i+1\over 2}}\\
    &\ge  C_{\bd}(\|\bv\|_1-s)^{{d+m\over 2}-1}\prod_{i=1}^mv_i^{-{d_i+1\over 2}},
    \end{align*}
    where in the third line we used \eqref{eq:EpsilonDefinition} and \eqref{eq:EpsilonBound} and in the last line we applied the fact that $m\ge 2$ and, hence, $2-5/(2m)>0$. This finishes the proof of the lemma.
\end{proof}

\section{Proofs of the main results}\label{sec:Proofs}

The main body of this section is devoted to the proof of our main Theorem \ref{thm:main}, the proof of Corollary \ref{cor:main} will be the content of the final Section \ref{sec:ProofCorollary}. 

Recall that $\bd=(d_1,\dotsc,d_m)$ is a sequence of positive integers and $\bbeta=(\beta_1,\dotsc,\beta_m)\in [0,\infty)^m$.
Furthermore, we set 
\begin{equation*}
    d:=\sum_{i=1}^m d_i \quad \text{ and }\quad \beta:=\sum_{i=1}^m \beta_i\geq 0, 
\end{equation*}
and recall that the \emph{beta-adjusted dimensions} $\mathbbm{k}=(k_0,\dotsc,k_m)$ where defined as
\begin{equation*}
    k_i := \frac{d_i+\beta_i}{1+\beta_i} \geq 1 \qquad \text{for $i\in\{1,\dotsc,m\}$},
\end{equation*}
and where we additionally set
\begin{equation*}
    k_0 := \frac{d+\beta}{m+\beta} = \sum_{i=1}^m t_i k_i \geq 1,
\end{equation*}
for $t_i:= \frac{1+\beta_i}{m+\beta}\in(0,1]$ and $\sum_{i=1}^m t_i=1$.

\subsection{Initial preparation}\label{sec:setup}

We start by applying the classical ideas going back to Rényi and Sulanke \cite{RenyiSulanke1,RenyiSulanke2} for investigating the convex hull of a random point set. 
Since with probability one the random polytope $\cP_{n,\bd}^{\bbeta}=\mathrm{co}\{\bX_i,\dotsc,\bX_n\}$ is simplicial, the facet number can almost surely be represented as
\begin{equation*}
    f_{d-1}(\cP_{n,\bd}^{\bbeta}) 
    = \sum_{1\leq j_1 \leq \dots \leq j_d\leq n} 
        \mathbf{1}\left\{\text{$\mathrm{co}\{\bX_{j_1},\dotsc,\bX_{j_d}\}$ is a facet of $\cP_{n,\bd}^{\bbeta}$}\right\}.
\end{equation*}
Using the fact that the random points $\bX_i$ are all identically distributed we find
\begin{align*}
    \EE f_{d-1}(\cP_{n,\bd}^{\bbeta}) 
        &= \binom{n}{d} \int_{(Z_{\bd})^{d}} \PP(\text{$\mathrm{co}\{\bx_1,\dotsc,\bx_d\}$ is a facet ...}\\ &\hspace{4cm} \text{... of $\mathrm{co}\{\bx_1,\dotsc,\bx_d,\bX_{d+1},\dotsc,\bX_{n}\}$})
            \, \prod_{i=1}^d \PP_{d}( \dint \bx_i;\bbeta,\bd).
\end{align*}
The points $\bx_1,\dotsc,\bx_d$ determine with probability one a hyperplane $H(\bw,s)$ with normal direction $\bw=\bw(\bx_1,\dotsc,\bx_d)$ and distance $s=s(\bx_1,\dotsc,\bx_d)$ from the origin.
Note that $\mathrm{co}\{\bx_1,\dotsc,\bx_d\}$ is a facet of $\mathrm{co}\{\bx_1,\dotsc,\bx_d,\bX_{d+1},\dotsc,\bX_n\}$ if and only if the points $\bX_{d+1},\dotsc,\bX_{n}$ are either contained in the half-space $H^+(\bw,s)$ or in $H^-(\bw,s)$. Using this observation we can now apply the affine Blaschke--Petkantschin formula we rephrased in Proposition \ref{prop:BP}. Recall that
\begin{equation*}
    \vol_{d-1}(Z_{\bd}\cap H(\bw,s);\bbeta,\bd) = \int_{Z_{\bd}\cap H(\bw,s)} f_{\bbeta,\bd}(\bx)\, \dint_{H(\bw,s)}\bx,
\end{equation*}
where $f_{\bbeta,\bd}$ is the block-beta density, i.e.,
\begin{equation*}
    f_{\bd,\bbeta}(\bx) = f_{\bd,\bbeta}((\by_1,\dotsc,\by_m)) = \prod_{i=1}^m f_{d_i,\beta_i}(\by_i),\qquad \bx=(\by_1,\dotsc,\by_m)\in \RR^{d_1}\times\dots\times\RR^{d_m},
\end{equation*}
see \eqref{eqn:beta_density}.
We derive
\begin{align*}
    \EE f_{d-1}(\cP_{n,\bd}^{\bbeta})
        &=\frac{(d-1)!}{2} \binom{n}{d} \int_{\SS^{d-1}} \int_{\RR} 
            \left[\left(\PP_{d}(Z_{\bd}\cap H^+(\bw,s);\bbeta,\bd\right)^{n-d} + \left(\PP_{d}(Z_{\bd}\cap H^-(\bw,s);\bbeta,\bd)\right)^{n-d}\right]  \\
        &\qquad\qquad\qquad \times \vol_{d-1}(Z_{\bd}\cap H(\bw,s)) \vol_{d-1}(Z_{\bd}\cap H(\bw,s);\bbeta,\bd)^d G(\bw,s)
            \, \dint s \, \sigma_{d-1}(\dint \bw)\\
        &=(d-1)! \binom{n}{d} \int_{\SS^{d-1}} \int_{\RR} 
            \left(1-\PP_{d}(Z_{\bd}\cap H^+(\bw,s);\bbeta,\bd)\right)^{n-d} \\
        &\qquad\qquad\qquad \times \vol_{d-1}(Z_{\bd}\cap H(\bw,s)) \vol_{d-1}(Z_{\bd}\cap H(\bw,s);\bbeta,\bd)^d
            G(\bw,s) \, \dint s \, \sigma_{d-1}(\dint \bw),
\end{align*}
where we recall that $G(\bw,s)$ stands for the Sylvester-functional we introduced at \eqref{eq:SylvesterFunctional}. Here, in the last step we used that $H^+(\bw,s) = H^-(-\bw,-s)$.

Finally, we recall that the cap $Z_{\bd}\cap H^+(\bw,s)$ and the section $Z_{\bd}\cap H(\bw,s)$ are up to ${\rm SO}(\bd)$-symmetry determined by the associated \emph{meta-cap} and \emph{meta-section}, see Lemma \ref{lem:ReductionLemma}.
We therefore decompose $\bw\in \SS^{d-1}$ into $\bu_i\in\SS^{d_i-1}$ using polyspherical coordinates, see Proposition \ref{prop:spherical_fubini}.
Since $Z_{\bd}$ is $\mathrm{SO}(\bd)$-symmetric and the block-beta distribution is $\mathrm{SO}(\bd)$-invariant, we find that for any choice of $\bu_i\in\SS^{d_i-1}$ we have for $\bw=\sum_{i=1}^m v_i\bu_i$ that $C^+(\bv,s)$ and $C(\bv,s)$ determine $Z_{\bd}\cap H^+(\bw,s)$ and $Z_{\bd}\cap H(\bw,s)$ for $\bv\in\SS^{m-1}_+$ up to some rotation from $\mathrm{SO}(\bd)$ (see Section \ref{sec:MetaObjects}). Thus, by Lemma \ref{lem:ReductionLemma} and Proposition \ref{prop:spherical_fubini} we get
\begin{align*}
    \EE f_{d-1}(\cP_{n,\bd}^{\bbeta})
        &\asymp n^d \int_{\SS^{m-1}_+}  \left(\prod_{i=1}^m v_i^{d_i-1}\right) \int_{\RR} 
            \left(1-C_{\bbeta,\bd}\,\PP_{m}(C^+(\bv,s);\tilde{\bbeta})\right)^{n-d} \\
        &\qquad \times \vol_{m-1}(C(\bv,s);(\bd-\bone)/2) \vol_{m-1}(C(\bv,s);\tilde{\bbeta})^d G(\bv,s)
          \, \dint s \,  \sigma_{m-1}(\dint \bv),
\end{align*}
where we set $\tilde{\bbeta}:=(\bd-\bone)/2+\bbeta \in [0,\infty)^m$ and write $G(\bv,s)$ for the value of $G(\bw,s)$ for $\bw=\sum_{i=1}^m v_i \bu_i$ for any choice of $\bu_i\in\SS^{d_i-1}$, i.e., such that $[\bw]_{\mathrm{SO}(\bd)}\cong\bv$ (see Section \ref{sec:Sylvester}).

\smallskip

We are now ready to proceed with our geometric estimates.

\subsection{Step 0: The case \texorpdfstring{$m=1$}{m=1}}

Note that for $m=1$ instead of polyspherical coordinates in $\RR^d$ we just use standard spherical coordinates, and by the radial symmetry of the $\beta$-distribution we find
\begin{align*}
    \EE f_{d-1}(\cP_{n,d}^{\beta})
        &\asymp n^d \int_{-1}^1 
            \left(1-C_{\beta,d}\,\PP_{1}([s,1];\tilde{\beta})\right)^{n-d} 
         \vol_{0}(s;(d-1)/2) \vol_{0}(s;\tilde{\beta})^d G(\be_1,s)
          \, \dint s,
\end{align*}
where $\tilde{\beta}=(d-1)/2+\beta\in[0,\infty)$ and we recall that by \eqref{eq:ZeroVolumeDef},
\begin{equation*}
    \vol_0(s;\beta) := c_{\beta} (1-s^2)^\beta \quad \text{for $s\in[-1,1]$}.
\end{equation*}
Now, for $s<0$ we have $\PP_{1}([s,1];\tilde{\beta}) > \frac{1}{2}$, which yields
\begin{align*}
    n^d \int_{-1}^0 (1-\PP_{1}([s,1];\tilde{\beta}))^{n-d} \vol_0(s;(d-1)/2) \vol_0(s;\tilde{\beta}) G(\be_1,s) \, \dint s
    \lesssim \frac{n^d}{2^{n-d}} = e^{-O(n)} \qquad \text{for $n\to\infty$}.
\end{align*}
It follows that 
\begin{equation}\label{eqn:step_m1}
    \EE f_{d-1}(\cP_{n,d}^\beta) \asymp n^d \int_{0}^1 (1-\PP_{1}([s,1];\tilde{\beta}))^{n-d} \vol_0(s;(d-1)/2) \vol_0(s;\tilde{\beta}) G(\be_1,s) \, \dint s + e^{-O(n)}.
\end{equation}
Next, we apply Lemma \ref{lem:AreaVolumeCapEstimates} to \eqref{eqn:step_m1}, estimate $G(\be_1,s)$ by Lemma \ref{lem:SylvesterFunctional} and substitute $h(s)=1-s$ to obtain
\begin{align*}
    \EE f_{d-1}(\cP_{n,d}^\beta) 
        &\asymp n^d \int_{0}^1 \left(1-C_{\beta,d} h^{\frac{d+1}{2}+\beta} \right)^{n-d}
            h^{\frac{d+1}{2}} \left(h^{\frac{d+1}{2}+\beta}\right)^d h^{-(d+1)} \, \dint h + e^{-O(n)}.
\end{align*}
Finally, substituting $z(h) = n C_{\beta,d}h^{\frac{d+1}{2}+\beta}$, yields
\begin{align*}
    \EE f_{d-1}(\cP_{n,d}^\beta) 
        \asymp n^{\frac{d-1}{d+2\beta+1}} \int_{0}^{nC_{\beta, d}} (1-z/n)^{n-d} z^{d - 1 - \frac{d-1}{d+2\beta+1}}\,\dint z + e^{-O(n)}
        \asymp n^{\frac{k_0-1}{k_0+1}},
\end{align*}
for $k_0 =k_1 = \frac{d+\beta}{1+\beta}$. This completes the proof of Theorem \ref{thm:main} for the case $m=1$. In what follows we thus assume that $m\geq 2$.

\begin{remark}
    Note that our arguments for $m=1$ are valid as long as $\tilde{\beta} = (d-1)/2+\beta\geq 0$, which in particular implies that for $d\geq 3$ we may choose $\beta>-1$.
\end{remark}

\subsection{Step 1: Reduction to a corner of the meta-cube}

We recall that the support function of a convex body $K\subset\RR^d$ is given by $h(K,\bu):=\max\limits_{\bx\in K}\bx\cdot\bu$, $\bu\in\SS^{d-1}$. It describes the signed distance from the origin to a supporting hyperplane of $K$ with unit normal direction $\bu$ and uniquely characterizes the convex body $K$, see \cite[Sec.\ 1.7]{Schneider:2014}. The support function of the product body $Z_{\bd}$ is given by
\begin{equation*}
    h(Z_{\bd},\bw) = \max_{\bx\in Z_{\bd}} \bx\cdot \bw = \sum_{i=1}^m h(B_2^{d_i},\bw_i) = \sum_{i=1}^m \|\bw_i\|_2.
\end{equation*}
Thus for $\bv\in\SS^{m-1}$ and $\bw\in\SS^{d-1}$ such that $[\bw]_{\mathrm{SO}(\bd)}\cong \bv$, i.e., $\bw=\sum_{i=1}^m v_i \bu_i$ for $\bu_i\in\SS^{d_i-1}$, we have 
\begin{equation*}
    Z_{\bd}\cap H(\bw,s) \neq \emptyset 
    \quad \Longleftrightarrow \quad 
    C(\bv,s)\neq \emptyset 
    \quad \Longleftrightarrow\quad 
    |s|\leq h(Z_{\bd},\bw) = \sum_{i=1}^m |v_i|\|\bu_i\|_2 = \|\bv\|_1.
\end{equation*}
Then,
\begin{align*}
    \EE f_{d-1}(\cP_{n,\bd}^{\bbeta})
        &\asymp n^d \int_{\SS^{m-1}_+} \left(\prod_{i=1}^m v_i^{d_i-1}\right) \int_{-\|\bv\|_1}^{\|\bv\|_1} 
            \left(1-C_{\bbeta,\bd}\,\PP_{m}(C^+(\bv,s);\tilde{\bbeta})\right)^{n-d} \\
        &\qquad  \times \vol_{m-1}(C(\bv,s);(\bd-\bone)/2) \vol_{m-1}(C(\bv,s);\tilde{\bbeta})^d G(\bv,s)
        \, \dint s \,  \sigma_{m-1}(\dint \bv).
\end{align*}
Next, we substitute $h(s) = (\|\bv\|_1-s)/2$, and obtain
\begin{align*}
    \EE f_{d-1}(\cP_{n,\bd}^{\bbeta})
        &\asymp n^d \int_{\SS^{m-1}_+} \left(\prod_{i=1}^m v_i^{d_i-1}\right) \int_{0}^{\|\bv\|_1} 
            \left(1-C_{\bbeta,\bd}\,\PP_{m}(C^+(\bv,\|\bv\|_1-2h);\tilde{\bbeta})\right)^{n-d} \\
        &\,  \times \vol_{m-1}(C(\bv,\|\bv\|_1-2h);(\bd-\bone)/2) \vol_{m-1}(C(\bv,\|\bv\|_1-2h);\tilde{\bbeta})^d G(\bv,s)
        \, \dint h \,  \sigma_{m-1}(\dint \bv).
\end{align*}
Now, for $\bv\in\SS^{m-1}_+$ if $h>\|\bv\|_\infty = \max_{i\in\{1,\dotsc,m\}} v_i$, then $C^+(\bv,\|\bv\|_1-2h)$ contains the simplex spanned by the vertices $\bz_0:=\bone$ and $\bz_i:=\bone-2\be_i$ for $i\in \{1,\dotsc,m\}$, since
\begin{equation*}
    \bv\cdot \bz_i = \|\bv\|_1-2v_i > \|\bv\|_1-2 h.
\end{equation*}
Thus, $\PP_{m}(C^+(\bv,\|\bv\|_1-2h);\tilde{\bbeta}) \geq \PP_{m}(\mathrm{co}\{\bz_0,\dotsc,\bz_m\};\tilde{\bbeta})=:C_2\in (0,1)$ for all $h\in[\|\bv\|_\infty, \|\bv\|_1]$. Using this and since we can upper bound $\vol_{m-1}(C(\bv,\|\bv\|_1-2h);(\bd-\bone)/2)$ and $\vol_{m-1}(C(\bv,\|\bv\|_1-2h);\tilde{\bbeta})^d$ by a constant that only depends on $\bd$ and $\bbeta$, see Lemma \ref{lem:AreaVolumeCapEstimates}, and since $G(\bv,\|\bv\|_1-2h)\leq 1$, we conclude that
\begin{align*}
    &n^d \int_{\SS^{m-1}_+} \left(\prod_{i=1}^m v_i^{d_i-1}\right) \int_{\|\bv\|_\infty}^{\|\bv\|_1} 
            \left(1-C_{\bbeta,\bd}\,\PP_{m}(C^+(\bv,\|\bv\|_1-2h);\tilde{\bbeta})\right)^{n-d} \\
        &\qquad  \times \vol_{m-1}(C(\bv,\|\bv\|_1-2h);(\bd-\bone)/2) \vol_{m-1}(C(\bv,\|\bv\|_1-2h);\tilde{\bbeta})^d G(\bv,s)
        \, \dint h \,  \sigma_{m-1}(\dint \bv)\\
    &\qquad\qquad\qquad\qquad\qquad\qquad\qquad\qquad\qquad \leq C_{\bbeta,\bd} n^d (1-C_2)^{n-d} = e^{-O(n)} \qquad \text{for $n\to\infty$}.
\end{align*}
This yields
\begin{align}\notag
    &\EE f_{d-1}(\cP_{n,\bd}^{\bbeta})\\\notag
        &\qquad \asymp n^d \int\limits_{\SS^{m-1}_+} \left(\prod_{i=1}^m v_i^{d_i-1}\right) \int\limits_{0}^{\|\bv\|_\infty} 
            \left(1-C_{\bbeta,\bd}\,\PP_{m}(C^+(\bv,\|\bv\|_1-2h);\tilde{\bbeta})\right)^{n-d} G(\bv,\|\bv\|_1-2h)\\\notag
        &\qquad\qquad \times \vol_{m-1}(C(\bv,\|\bv\|_1-2h);(\bd-\bone)/2) \vol_{m-1}(C(\bv,\|\bv\|_1-2h);\tilde{\bbeta})^d 
        \, \dint s \,  \sigma_{m-1}(\dint \bv) + e^{-O(n)}\\\notag
        &\qquad\asymp n^d \int\limits_{\SS^{m-1}_+} 
            \left(\prod_{i=1}^m v_i^{d_i-1}\right) \int\limits_{0}^{\|\bv\|_\infty}\left(1-C_{\bbeta,\bd}\,\PP_{m}(C^+(\bv,\|\bv\|_1-2h);\tilde{\bbeta})\right)^{n-d} G(\bv,\|\bv\|_1-2h)\\\label{eqn:proof_step1}
        &\qquad\qquad \times 
            \vol_{m-1}(C(\bv,\|\bv\|_1-2h);(\bd-\bone)/2) \vol_{m-1}(C(\bv,\|\bv\|_1-2h);\tilde{\bbeta})^d
                \, \dint s \,  \sigma_{m-1}(\dint \bv),
\end{align}
where we drop the error term $e^{-O(n)}$, because we will see in the end that it will be dominated by the asymptotic order in $n$ of the remaining term.

\subsection{Step 2: Upper bound}

Since $G(\bv,\,\cdot\,)\leq 1$ and by applying Lemma \ref{lem:AreaVolumeCapEstimates} on \eqref{eqn:proof_step1} we derive
\begin{align*}
    \EE f_{d-1}(\cP_{n,\bd}^{\bbeta}) 
        &\lesssim n^d \int_{\SS^{m-1}_+} \int\limits_{0}^{\|\bv\|_\infty} 
            \left(1 - C_{\bbeta,\bd} \prod_{i=1}^m \min \left\{ \frac{h}{v_i}, 1 \right\}^{\frac{d_i+1}{2}+\beta_i}\right)^{n-d} \\
        &\quad \times
            \left(\prod_{i=1}^m \min \left\{ \frac{h}{v_i}, 1 \right\}^{\frac{d_i+1}{2}+\beta_i}\right)^d
            \left(\prod_{i=1}^m \min \left\{ \frac{h}{v_i}, 1 \right\}^{\frac{d_i+1}{2}}\right)
            h^{-(d+1)}\left(\prod_{i=1}^m v_i^{d_i-1}\right)\, \dint h\, \sigma_{m-1}(\dint\bv).
\end{align*}

Next, we write $\Sigma(m)$ for the set of permutations of $\{1,\dotsc,m\}$ and decompose $\SS^{m-1}_+$ into $m!$ open subsets by assuming that $1>v_{\tau(1)}> \dotsc > v_{\tau(m)}>0$ for some $\tau\in\Sigma(m)$. Note that the set of vectors $\bv\in\SS^{m-1}_+$ that are not contained in the union of these open subsets have measure zero with respect to $\sigma_{m-1}$ and can therefore be ignored.
Thus,
\begin{align}\label{eq:f-vector-Phi}
    \EE f_{d-1}(\cP_{n,\bd}^{\bbeta}) 
        &\lesssim \sum_{\tau\in\Sigma(m)} \Phi_\tau(n),
\end{align}
where we set
\begin{align*}
    \Phi_\tau(n) &:= n^d \int\limits_{\substack{\bv\in\SS^{m-1}_+\\1>v_{\tau(1)}> \dotsc> v_{\tau(m)}>0}}  \int\limits_{0}^{v_{\tau(1)}} 
        \left(1 - C_{\bbeta,\bd} \prod_{i=1}^m \min \left\{ \frac{h}{v_i}, 1 \right\}^{\frac{d_i+1}{2}+\beta_i}\right)^{n-d} \\
        &\qquad\qquad \times
            \left(\prod_{i=1}^m \min \left\{ \frac{h}{v_i}, 1 \right\}^{(d+1)\frac{d_i+1}{2}+d\beta_i}\right)
            h^{-(d+1)}\left(\prod_{i=1}^m v_i^{d_i-1}\right)\, \dint h\, \sigma_{m-1}(\dint\bv).
\end{align*}
and we just write $\Phi(n)$ for $\tau=\mathrm{id}$.

Now, assuming w.l.o.g.\ $1>v_1> \dotsc > v_m>0$, we have that $1 = \sum_{i=1}^m v_i^2 < m\, v_1^2$, which yields $1> v_1> \frac{1}{\sqrt{m}}$. We parametrize the spherical simplex via $v_1=v_1(v_2,\dotsc,v_m)=\sqrt{1-\sum_{i=2}^m v_i^2}$, which has the Jacobian
\begin{equation*}
    \frac{\dint \sigma_{m-1}}{\dint (v_2,\dotsc,v_m)} = \sqrt{1+\|\nabla v_1\|^2} = \frac{1}{v_1}.
\end{equation*}
Thus,
\begin{align*}
    \Phi(n)
        &=n^d  
        \idotsint\limits_{\substack{1> v_1> \dotsc> v_m> 0\\ v_2^2+\ldots +v_m^2\leq 1}} \int_{0}^{v_1}
         g(h,v_2,\dotsc,v_m) \, \dint h\, \left(\prod_{i=2}^{m}\dint v_i\right),
\end{align*}
where we set
\begin{align*}
    g(h,v_2,\dotsc,v_m) &:= \left(1 - C_{\bbeta,\bd}\Big({h\over v_1}\Big)^{\frac{d_1+1}{2}+\beta_1} \prod_{i=2}^m \min \left\{ \frac{h}{v_i}, 1 \right\}^{\frac{d_i+1}{2}+\beta_i}\right)^{n-d}\\
        &\qquad \times
            \left(\Big({h\over v_1}\Big)^{(d+1)\frac{d_1+1}{2}+d\beta_1}\prod_{i=2}^m \min \left\{ \frac{h}{v_i}, 1 \right\}^{(d+1)\frac{d_i+1}{2}+d\beta_i}\right)
            h^{-(d+1)} v_1^{-1}\left(\prod_{i=1}^m v_i^{d_i-1}\right).
\end{align*}
Introducing the change of variables $\tilde h=h/v_1$ and $\tilde v_i=v_i/v_1$, $i\in\{2,\ldots,m\}$ with Jacobian $v_1^{-m}$ and noting that for $1>\tilde v_2 >\ldots> \tilde v_{m}>0$ we have $\tilde v_2^2+\ldots+\tilde v_{m}^2< m-1< v_1^{-2}-1$ and that $v_1^{-2}\asymp 1$ we further get
\begin{align*}
    \Phi(n) &\asymp n^d \int\limits_{1>\tilde v_2> \dotsc> \tilde v_m>0}  \int\limits_{0}^{1} 
        \tilde g(\tilde h,\tilde v_2,\dotsc,\tilde v_m) \, \dint \tilde h\, \left(\prod_{i=2}^{m}\dint \tilde v_i\right),
\end{align*}
where 
\begin{align*}
    \tilde g(\tilde h,\tilde v_2,\dotsc,\tilde v_m) &:= \left(1 - C_{\bbeta,\bd}\tilde h^{\frac{d_1+1}{2}+\beta_1} \prod_{i=2}^m \min \left\{ \frac{ \tilde h}{ \tilde v_i}, 1 \right\}^{\frac{d_i+1}{2}+\beta_i}\right)^{n-d}\\
        &\qquad \qquad \times
            \left(\tilde h^{(d+1)\frac{d_1+1}{2}+d\beta_1}\prod_{i=2}^m \min \left\{ \frac{\tilde h}{\tilde v_i}, 1 \right\}^{(d+1)\frac{d_i+1}{2}+d\beta_i}\right)
            \tilde h^{-(d+1)}\left(\prod_{i=2}^m  \tilde v_i^{d_i-1}\right).
\end{align*}

Next split the domain of integration and set, for $j\in\{2,\dotsc, m-1\}$,
\begin{align*}
    \Phi^j(n) &:= n^d \idotsint\limits_{0<\tilde v_m<\dotsc <\tilde v_2< 1} \int_{\tilde v_{j+1}}^{\tilde v_j}
         \tilde g(\tilde h,\tilde v_2,\dotsc,\tilde v_m)\, \dint \tilde h\, \left(\prod_{i=2}^{m}\dint \tilde v_i\right),
 \intertext{and}        
 \Phi^{m}(n) &:= n^d \idotsint\limits_{0<\tilde v_m< \dotsc <  \tilde v_2< 1} \int_{0}^{\tilde v_m}
         \tilde g(\tilde h,\tilde v_2,\dotsc,\tilde v_m) \, \dint \tilde h\, \left(\prod_{i=2}^{m}\dint \tilde v_i\right),\\
\Phi^{1}(n) &:= n^d 
        \idotsint\limits_{0<\tilde v_m<\dotsc < \tilde  v_2< 1} \int_{\tilde v_2}^{1}
        \tilde g(\tilde h,\tilde v_2,\dotsc,\tilde v_m) \, \dint \tilde h\, \left(\prod_{i=2}^{m}\dint \tilde v_i\right).
\end{align*}
We may treat all cases simultaneously by assuming $\tilde v_1:=1$ and $\tilde v_{m+1}:=0$. Then
\begin{equation*}
    \Phi(n) \asymp \sum_{j=1}^m \Phi^j(n).
\end{equation*}
Similarly, we define the integrals $\Phi_{\tau}^{j}(n)$ for any $\tau\in\Sigma(m)$ by replacing the indices $2,\ldots, m$ by $\tau(2),\ldots,\tau(m)$.

Let us briefly remark here that for $\Phi^j(n)$ we have
\begin{equation*}
    0=\tilde v_{m+1}< \dotsc <\tilde v_{j+1} < \tilde h < \tilde v_j < \dots< \tilde v_1=1,
\end{equation*}
which, since $v_1>0$, is equivalent to 
\begin{equation*}
    0=\tilde v_{m+1}< v_m< \dotsc < v_{j+1} <   h < v_{j} <  \dotsc<   v_1\leq 1.
\end{equation*}
This relates to the case where the hyperplane $H(\bv,\|\bv\|_1-2h)$ strictly separates the vertices $\bz_0=\bone$ and $\bz_i=\bone-2\be_i$, $i\in\{j+1,\ldots, m\}$ from the vertices $\{\bz_i: i\in\{1,\dotsc, j\}\}$, i.e., $i\leq j$ is equivalent to
\begin{equation*}
    \|\bv\|_1 - 2h > \|\bv\|_1 - 2v_i = \bv\cdot \bz_i,
\end{equation*}
which is in turn equivalent to $\bz_i\not\in C^+(\bv,\|\bv\|_1-2h)$. Thus, in this case the meta-cap $C^+(\bv,\|\bv\|_1-2h)$ contains 
the $(m-j)$-simplex $\mathrm{co}\,\{\bz_0,\bz_{j+1},\bz_{j+2},\dotsc,\bz_m\}$.\

We shall show now that we can eliminate all coordinates $\tilde v_i$ for $i\ge j+1$. Notice that for $\tilde h\in(\tilde v_{j+1},\tilde v_j)$ we have for all $i\ge  j+1$,
\begin{equation*}
    \frac{\tilde h}{\tilde v_{i}} \geq \frac{\tilde v_{j+1}}{\tilde v_i} \geq 1,
\end{equation*}
and for all $i\leq j$ we have
\begin{equation*}
    \frac{\tilde h}{\tilde v_i} \leq \frac{\tilde v_j}{\tilde v_i} \leq 1.
\end{equation*}
Let us introduce the notation 
\[
d(j):=\sum_{i=1}^jd_i\qquad\qquad \text{and}\qquad\qquad \beta(j):=\sum_{i=1}^j\beta_i.
\]
Using Fubini's theorem to first integrate over the coordinates $\tilde v_{j+1}, \dotsc, \tilde v_m$, we estimate
\begin{align*}
    \Phi^j(n) &\lesssim n^d 
        \idotsint\limits_{0<\tilde h<\tilde v_j<\dotsc < \tilde v_1= 1}
         \left(1 - C_{\bbeta,\bd} \tilde h^{\frac{d_1+1}{2}+\beta_1} \prod_{i=2}^j \left(\frac{\tilde h}{\tilde v_i}\right)^{\frac{d_i+1}{2}+\beta_i}\right)^{n-d} \\
        &\qquad \times
            \left(\tilde h^{(d+1)\frac{d_1+1}{2}+d\beta_1}\prod_{i=2}^j \left(\frac{\tilde h}{\tilde v_i}\right)^{(d+1)\frac{d_i+1}{2}+d\beta_i}\right)
            \tilde h^{-(d+1)} \\
        &\qquad\times\left(\prod_{i=j+1}^m \int_{0}^{\tilde h}  \tilde v_i^{d_i-1} \dint \tilde v_i\right)\, \left(\prod_{i=2}^{j}\tilde  v_i^{d_i-1}\dint \tilde v_i\right) \dint \tilde h\\
        &= n^d
            \idotsint\limits_{0<\tilde h<\tilde v_j<\dotsc < \tilde v_1= 1}
         \left(1 - C_{\bbeta,\bd} \tilde h^{\frac{d_1+1}{2}+\beta_1} \prod_{i=2}^j \left(\frac{\tilde h}{\tilde v_i}\right)^{\frac{d_i+1}{2}+\beta_i}\right)^{n-d} \\
        &\qquad \times
            \left(\tilde h^{(d+1)\frac{d_1+1}{2}+d\beta_1}\prod_{i=2}^j \left(\frac{\tilde h}{\tilde v_i}\right)^{(d+1)\frac{d_i+1}{2}+d\beta_i}\right)
            \tilde h^{-1-d(j)}
            \, \left(\prod_{i=2}^{j} \tilde v_i^{d_i-1}\dint \tilde v_i\right) \dint \tilde h,
\end{align*}
where we interpret the empty product as $1$. Using the substitution $z(\tilde h) = \tilde h/\tilde v_j$ we find that
\begin{align*}
    \Phi^j(n) &\lesssim n^d
            \idotsint\limits_{0<\tilde v_j<\dotsc < \tilde v_1= 1} \int_{0}^1
             \left(1 - C_{\bbeta,\bd} z^{\frac{d(j)+j}{2}+\beta(j)} \tilde v_j^{\frac{d_1+1}{2}+\beta_1} \prod_{i=2}^{j-1} \left(\frac{\tilde v_j}{\tilde v_i}\right)^{\frac{d_i+1}{2}+\beta_i}\right)^{n-d} \\
                     &\qquad \times
            \left(z^{(d+1)\frac{d(j)+j}{2}+d\beta(j)} \tilde v_j^{(d+1)\frac{d_1+1}{2}+d\beta_1} 
                \prod_{i=2}^{j-1} \left(\frac{\tilde v_j}{\tilde v_i}\right)^{(d+1)\frac{d_i+1}{2}+d\beta_i}\right)\\
        &\qquad \times 
            z^{-1-d(j)} \tilde v_j^{-j-(d_1-1)} \,\dint z\, \left(\prod_{i=2}^{j-1} \left(\frac{\tilde v_j}{\tilde v_i}\right)^{-(d_i-1)}\dint \tilde v_i\right)\dint \tilde v_j.
\end{align*}
Next, we substitute $u_i(\tilde v_i)=\tilde v_j/\tilde v_i< 1$ for $i\in\{2,\dotsc,j-1\}$, $u_1(\tilde v_j)=\tilde v_j < 1$, and arrive at
\begin{align*}
    \Phi^j(n) &\lesssim n^d
            \idotsint\limits_{0<u_1<\dotsc < u_{j-1}< 1} \int_{0}^1
            \left(1 - C_{\bbeta,\bd} z^{\frac{d(j)+j}{2}+\beta(j)} \prod_{i=1}^{j-1} u_i^{\frac{d_i+1}{2}+\beta_i}\right)^{n-d} \\
        &\qquad \times
            \left(z^{(d+1)\frac{d(j)+j}{2}+d\beta(j)} \prod_{i=1}^{j-1} u_i^{(d+1)\frac{d_i+1}{2}+d\beta_i}\right)
            z^{-1-d(j)} \,\dint z\, \left(\prod_{i=1}^{j-1} u_i^{-(d_i+1)}\dint u_i\right).
\end{align*}
Since all our arguments can be applied similarly to any permutation $\tau\in\Sigma(j-1)$, i.e., such that $\tau(i)=i$ for all $i\in\{j,\dotsc,m\}$, we also find that
\begin{align*}
    \sum_{\tau\in\Sigma(j-1)}\Phi^j_\tau(n) &\lesssim n^d
            \int_{[0,1]^j}\left(1 - C_{\bbeta,\bd} z^{\frac{d(j)+j}{2}+\beta(j)} \prod_{i=1}^{j-1} u_i^{\frac{d_i+1}{2}+\beta_i}\right)^{n-d} \\
        &\qquad \times
            \left(z^{(d-1)\left(\frac{d(j)+j}{2}+\beta(j)\right) +j-1+\beta(j)} \prod_{i=1}^{j-1} u_i^{(d-1)\big(\frac{d_i+1}{2}+\beta_i\big) + \beta_i}\right)
             \,\dint z\, \left(\prod_{i=1}^{j-1}\dint u_i\right).
\end{align*}
Finally, we substitute for $i\in\{1,\ldots,j-1\}$,
\begin{align}\notag
    x_i(u_i) &= u_i^{\frac{d_i+1}{2}+\beta_i} & \
        \frac{\dint x_i}{\dint u_i} &= C_{\bbeta,\bd}\, x_i^{-\frac{d_i+2\beta_i-1}{d_i+2\beta_i+1}}, \\
        \intertext{and}
    x_0(z) &= z^{\frac{d(j)+j}{2}+\beta(j)} & 
        \frac{\dint x_0}{\dint z} &= C_{\bbeta,\bd}\, x_0^{-1+\frac{2}{d(j)+2\beta(j)+j}},\label{eqn:x_subs}
\end{align}
and arrive at
\begin{align*}
 \sum_{\tau\in\Sigma(j-1)}\Phi^j_\tau(n) &\lesssim n^d
            \int_{[0,1]^j}
         \left(1 - C_{\bbeta,\bd} \prod_{i=0}^{j-1} x_i\right)^{n-d} 
            x_0^{d-1-\frac{d(j)-j}{d(j)+2\beta(j)+j}} \prod_{i=1}^{j-1} x_i^{d-1 - \frac{d_i-1}{d_i+2\beta_i+1}}
            \, \prod_{i=0}^{j-1} \dint x_i\\
            &= n^d
            \int_{[0,1]^j}
         \left(1 - C_{\bbeta,\bd} \prod_{i=0}^{j-1} x_i\right)^{n-d} 
                \prod_{i=0}^{j-1} x_i^{d-1 - \frac{k_i^j-1}{k_i^j+1}}
            \, \dint x_i,
\end{align*}
where we set
\begin{equation*}
    k_i^j := k_i = \frac{d_i+\beta_i}{\beta_i+1} \qquad \text{and}\qquad 
    k_0^j := \frac{\sum_{i=1}^j (d_i+\beta_i)}{\sum_{i=1}^j (\beta_i+1)} = \sum_{i=1}^j \frac{\beta_i+1}{\sum_{k=1}^j (\beta_k+1)} k_i.
\end{equation*}
By Lemma \ref{lm:AffentrangerWieacker} we conclude that
\begin{equation}\label{eqn:asymptotic_rate_bound}
    \sum_{\tau\in\Sigma(j-1)}\Phi^j_\tau(n)\lesssim n^{\frac{{k}_{\max}^j-1}{{k}_{\max}^j+1}} (\ln n)^{\#{k}_{\max}^j-1},
\end{equation}
with 
\[
    {k}_{\max}^j := \max_{i\in\{0,\dotsc,j-1\}} k_i^j \qquad \text{and}\qquad
    \#{k}_{\max}^j := \#\left\{i\in\{0,\dotsc,j-1\} : k_i^j={k}_{\max}^j\right\}.
\]
Thus, the $j$-tuple $(k_0^j,k_1,\dotsc,k_{j-1})$ determines the asymptotic rate in $n$ and it depends on the $j$-tuple $(k_1,\dotsc,k_j)$, since $k_0^j$ is a strict convex combination of $k_i$, $i\in\{1,\ldots,j\}$. Further, note that \eqref{eqn:asymptotic_rate_bound} also implies 
\begin{equation}\label{eqn:asymptotic_rate_bound2}
    \sum_{\tau\in\Sigma(m)}\Phi^j_\tau(n)\lesssim \max_{\tau\in\Sigma(m)} 
 n^{\frac{{k}_{\max}^{j,\tau}-1}{{k}_{\max}^{j,\tau}+1}} (\ln n)^{\#{k}_{\max}^{j,\tau}-1},
\end{equation}
where
\begin{align}
    {k}_{\max}^{j,\tau} &:=\max\{k_{\tau(1)},\ldots, k_{\tau(j-1)}, k_0^{j,\tau}\},\quad k_0^{j,\tau}=\sum_{i=1}^j \frac{\beta_{\tau(i)}+1}{\sum_{k=1}^j (\beta_{\tau(k)}+1)} k_{\tau(i)},\notag\\
    &\quad \text{and}\quad
    \#{k}_{\max}^{j,\tau} := \sum_{i=1}^{j-1}{\bf 1}\{ k_{\tau(i)}={k}_{\max}^{j,\tau}\}+{\bf 1}\{k_0^{j,\tau}={k}_{\max}^{j,\tau}\}.\label{eq:k_max}
\end{align}
We also note that ${k}_{\max}^{j,\tau}\leq {k}_{\max}^{j+1,\tau}$ for any $j\in\{1,\ldots,m-1\}$ and $\tau\in \Sigma(m)$. Indeed, we have $\max_{i\in\{1,\dotsc,j-1\}} k_{\tau(i)}\leq \max_{i\in\{1,\dotsc,j\}} k_{\tau(i)}$ and 
\[
k_0^{j,\tau}=\sum_{i=1}^j \frac{\beta_{\tau(i)}+1}{\sum_{k=1}^j (\beta_{\tau(k)}+1)} k_{\tau(i)}\leq \max_{i\in\{1,\dotsc,j\}} k_{\tau(i)}.
\]
At the same time, for any $j\in\{1,\ldots,m-1\}$ and $\tau\in \Sigma(m)$ if  ${k}_{\max}^{j,\tau}= {k}_{\max}^{j+1,\tau}$ we have $\#{k}_{\max}^{j,\tau}\leq \#{k}_{\max}^{j+1,\tau}$. In order to show this consider the following two cases. First, if $k_{\tau(j)}={k}_{\max}^{j,\tau}$, then
\[
 \#{k}_{\max}^{j+1,\tau}=\#{k}_{\max}^{j,\tau}+1+{\bf 1}\{k_0^{j+1,\tau}={k}_{\max}^{j,\tau}\}-{\bf 1}\{k_0^{j,\tau}={k}_{\max}^{j,\tau}\}\ge \#{k}_{\max}^{j,\tau}.
\]
On the other hand if $k_{\tau(j)}<{k}_{\max}^{j,\tau}$, then
\[
k_0^{j,\tau}=\sum_{i=1}^j \frac{\beta_{\tau(i)}+1}{\sum_{k=1}^j (\beta_{\tau(k)}+1)} k_{\tau(i)}<{k}_{\max}^{j,\tau},
\]
and in this case we also have
\[
 \#{k}_{\max}^{j+1,\tau}=\#{k}_{\max}^{j,\tau}+{\bf 1}\{k_0^{j+1,\tau}={k}_{\max}^{j,\tau}\}\ge \#{k}_{\max}^{j,\tau}.
\]
Combining these observations together we conclude that for any $j\in\{1,\ldots,m-1\}$ and $\tau\in \Sigma(m)$ it holds that
\[
n^{\frac{{k}_{\max}^{j,\tau}-1}{{k}_{\max}^{j,\tau}+1}} (\ln n)^{\#{k}_{\max}^{j,\tau}-1}\lesssim n^{\frac{{k}_{\max}^{j+1,\tau}-1}{{k}_{\max}^{j+1,\tau}+1}} (\ln n)^{\#{k}_{\max}^{j+1,\tau}-1}.
\]
By \eqref{eqn:asymptotic_rate_bound2}, for any $j\in\{1,\ldots,m\}$, this implies
\[
\sum_{\tau\in\Sigma(m)}\Phi^j_\tau(n)\lesssim \sum_{\tau\in\Sigma(m)}\Phi^{j+1}_\tau(n)\lesssim\ldots \lesssim \sum_{\tau\in\Sigma(m)}\Phi^m_\tau(n).
\]

Thus, by \eqref{eq:f-vector-Phi} we derive
\begin{equation*}
    \EE f_{d-1}(\cP_{n,\bd}^{\bbeta}) \lesssim \sum_{\tau\in\Sigma(m)} \Phi_\tau(n) 
    \lesssim \max_{\tau\in\Sigma(m)} n^{\frac{{k}_{\max}^\tau-1}{{k}_{\max}^\tau+1}} (\ln n)^{\#{k}_{\max}^\tau-1},
\end{equation*}
where
\begin{align}
    {k}_{\max}^\tau &= \max \{ k_0 , k_{\tau(1)} ,\dots,k_{\tau(m-1)}\},\quad k_0:={d+\beta\over \beta+m},\notag\\
    \text{and} \quad 
    &\#{k}_{\max}^\tau = \sum_{i=1}^{m-1}{\bf 1}\{k_{\tau(i)}={k}_{\max}^{\tau}\}+{\bf 1}\{k_0={k}_{\max}^{\tau}\}.\label{def:k_max}
\end{align}
This concludes the proof of the upper bound for $m\geq 2$.

\subsection{Step 3: Lower bound}

From Step 2 we expect to see the maximal asymptotic order in $n$ to appear for the case when a hyperplane $H(\bv,s)$ strictly separates $\bz_0=\bone$ from all other vertices $\bz_i=\bone-2\be_i$ for $i\in\{1,\dotsc,m\}$ of the meta-cube. 
To gain control over the ordering of the coordinates $v_i$ we again decompose $\SS^{m-1}_+$ and restrict to $h<\min_{i\in\{1,\dotsc,m\}} v_i$ in \eqref{eqn:proof_step1} to apply Lemma \ref{lem:CapVolume} and Lemma \ref{lm:CapArea}. This gives
\begin{equation*}
    \EE f_{d-1}(\cP_{n,\bd}^\bbeta) \gtrsim \sum_{\tau\in\Sigma(m)} \Psi_\tau(n),
\end{equation*}
where we set
\begin{align*}
    \Psi_\tau(n)  &:= \, n^d \int\limits_{\substack{\bv\in\SS^{m-1}_+\\ 1 > v_{\tau(1)} > \dotsc > v_{\tau(m)} >0}} \int_{0}^{v_{\tau(m)}}
        \left(1-C_{\bbeta,\bd} \prod_{i=1}^m \left(\frac{h}{v_{i}}\right)^{\frac{d_i+1}{2}+\beta_i}\right)^{n-d}\\
        &\qquad\qquad \times\left(\prod_{i=1}^m \left(\frac{h}{v_i}\right)^{(d+1)\frac{d_i+1}{2}+d\beta_i}\right) h^{-(d+1)} \left(\prod_{i=1}^m v_i^{d_i-1}\right)
            \, \dint h \, \sigma_{m-1}(\dint \bv).
\end{align*}
For simplicity we write $\Psi(n)$ if $\tau=\mathrm{id}$.

We assume without loss of generality that the coordinates of $\bv$ are order in such a way that $1> v_1>\dotsc> v_m>0$. 
Using the same argument as in Step 3, namely the substitutions $v_1=\sqrt{1-\sum_{i=2}^mv_i^2}$ and $\tilde h=h/v_1$, $\tilde v_i=v_i/v_1$, $i\in\{2,\ldots,m\}$ we get
\begin{align*}
    \Psi(n)  &= n^d \int\limits_{1 > \tilde v_2 > \dotsc > \tilde v_m >0} \int_{0}^{\tilde v_m}
        \left(1-C_{\bbeta,\bd} \tilde h^{\frac{d_1+1}{2}+\beta_1} \prod_{i=2}^m \left(\frac{\tilde h}{\tilde v_{i}}\right)^{\frac{d_i+1}{2}+\beta_i}\right)^{n-d}\\
        &\qquad\qquad \times\left(\tilde h^{(d+1)\frac{d_1+1}{2}+d\beta_1} \prod_{i=2}^m \left(\frac{\tilde h}{\tilde v_i}\right)^{(d+1)\frac{d_i+1}{2}+d\beta_i}\right) \tilde h^{-(d+1)} \left(\prod_{i=2}^m \tilde v_i^{d_i-1}\right)
            \, \dint \tilde h \left( \prod_{i=2}^{m} \dint \tilde v_i\right).
\end{align*}
Again, as in Step 3, for $\Phi^m(n)$ we now substitute $z(h) = \tilde h/\tilde v_m$ first and then $u_i(\tilde v_i)=\tilde v_m/\tilde v_i$ for $i\in\{2,\dotsc,m-1\}$ and $u_1(\tilde v_m)=\tilde v_m$. This way, we arrive at
\begin{align*}
    \Psi(n)  &\gtrsim \, n^d \int\limits_{1 > u_{m-1} > \dotsc > u_1 >0} \int_{0}^{1}
        \left(1-C_{\bbeta,\bd} z^{\frac{d+m}{2}+\beta} \prod_{i=1}^{m-1} u_i^{\frac{d_i+1}{2}+\beta_i}\right)^{n-d}\\
        &\qquad\qquad \times\left(z^{(d-1)\left(\frac{d+m}{2}+\beta\right)+m-1+\beta} \prod_{i=1}^{m-1} u_i^{(d-1)\left(\frac{d_i+1}{2}+\beta_i\right)+\beta_i}\right)
            \, \dint z \left(\prod_{i=1}^{m-1} \dint u_i\right).
\end{align*}
Combining these estimates for all $\tau\in\Sigma(m-1)$ yields
\begin{align*}
    \sum_{\tau\in\Sigma(m-1)} \Psi_\tau(n) &\gtrsim \, n^d \int\limits_{[0,1]^m}
        \left(1-C_{\bbeta,\bd} z^{\frac{d+m}{2}+\beta} \prod_{i=1}^{m-1} u_i^{\frac{d_i+1}{2}+\beta_i}\right)^{n-d}\\
        &\qquad\qquad \times\left(z^{(d-1)\left(\frac{d+m}{2}+\beta\right)+m-1+\beta} \prod_{i=1}^{m-1} u_i^{(d-1)\left(\frac{d_i+1}{2}+\beta_i\right)+\beta_i}\right)
            \, \dint z \left(\prod_{i=1}^{m-1} \dint u_i\right).
\end{align*}
Now, using the substitution \eqref{eqn:x_subs}, we calculate
\begin{align*}
    \sum_{\tau\in\Sigma(m-1)} \Psi_\tau(n) \gtrsim  n^d
            \int_{[0,1]^m}
         \left(1 - C_{\bbeta,\bd} \prod_{i=0}^{m-1} x_i\right)^{n-d} 
                \prod_{i=0}^{m-1} x_i^{d-1 - \frac{k_i-1}{k_i+1}}\, \dint x_i,
\end{align*}
where we recall that
\begin{equation*}
    k_0 = \frac{d+\beta}{\beta+m}.
\end{equation*}
We are now in the position to use Lemma \ref{lm:AffentrangerWieacker} to conclude that
\begin{align*}
    \EE f_{d-1}(\cP_{n,\bd}^\bbeta) \gtrsim \sum_{\tau\in\Sigma(m)} \Psi_\tau(n) \gtrsim \max_{\tau\in\Sigma(m)} n^{\frac{{k}_{\max}^\tau-1}{{k}_{\max}^\tau+1}} (\ln n)^{\#{k}_{\max}^\tau-1},
\end{align*}
where ${k}_{\max}^\tau$ and $\#{k}_{\max}^\tau$ are defined as in \eqref{def:k_max}.
This completes the proof of the lower bound.

\subsection{Step 4: Conclusion}\label{subsec:Conclusion}

Step 1--3 show that
\begin{align*}
    \EE f_{d-1}(\cP_{n,\bd}^{\bbeta}) 
        &\asymp \max_{\tau\in\Sigma(m)} n^{\frac{{k}_{\max}^\tau-1}{{k}_{\max}^\tau+1}} (\ln n)^{\#{k}_{\max}^\tau-1},
\end{align*}
where ${k}_{\max}^\tau$ and $\#{k}_{\max}^\tau$ are defined in \eqref{def:k_max}.

From this relation we see that the asymptotic growth rate of $\EE f_{d-1}(\cP_{n,\bd}^{\bbeta})$ as a function of $n$ is controlled by the $m$-tuple $(k_0,k_{\tau(1)},\dotsc,k_{\tau(m-1)})$. Applying a permutation $\tau\in\Sigma(m)$ on $\bd$ and $\bbeta$ may change the asymptotic rate, because we obtain all $m$-tuples $(k_0,k_1,\dotsc,\hat{k}_j,\dotsc,k_m)$, where $\hat{k}_j$ indicates that the corresponding term is dropped from the sequence. If we order $k_1=\dots=k_\ell>k_{\ell+1}\geq \dots\geq k_m$, then, for $\ell\neq m$, the dominating rate is $n^{\frac{k_1-1}{k_1+1}} (\ln n)^{\ell-1}$, which is realized, for example, by the sequence $(k_0,k_1,\dotsc,k_{m-1})$ and we have
\begin{equation}\label{eqn:kmax}
    \tilde{k}_{\max}=k_1=k_{\max} \quad \text{ and }\quad \#\tilde{k}_{\max}=\ell =\#k_{\max}.
\end{equation}
On the other hand, if we have $\ell = m$, then $k_0=\dotsc=k_m$ and for all sequences $(k_0,k_1,\dotsc,\hat{k}_j,\dotsc,k_m)$ we also have \eqref{eqn:kmax} and the maximal rate is $n^{\frac{k_1-1}{k_1+1}} (\ln n)^{m-1}$. This concludes the proof of Theorem \ref{thm:main}.\qed

\subsection{Proof of Corollary \ref{cor:main}}\label{sec:ProofCorollary}

Part (B) is immediate by combining \eqref{eq:FloatingBody for fj}  and \eqref{eq:FloatingBody for Vol} with Theorem \ref{thm:main} in the uniform case. To deduce (A) from Theorem \ref{thm:main} we first notice that almost surely $\cP_{n,\bd}^\bbeta$ is a simplicial polytope. In particular, each of its facets is almost surely a simplex of dimension $d-1$, which has precisely $\binom{d}{j+1}$ faces of dimension $j\in\{0,\ldots,d-1\}$. Since each $j$-dimensional faces of $\cP_{n,\bd}^\bbeta$ must be a $j$-face of at least one facet, we have $f_j(\cP_{n,\bd}^\bbeta)\leq \binom{d}{j+1}f_{d-1}(\cP_{n,\bd}^\bbeta)$. Putting $C_d:=\max_{j\in\{0,\ldots,d-1\}}\{\binom{d}{j+1}\}$ and taking expectations we thus conclude that
\begin{equation}\label{eq:fj<Cfd-1}
    \EE f_j(\cP_{n,\bd}^\bbeta) \leq C_d\EE f_{d-1}(\cP_{n,\bd}^\bbeta). 
\end{equation}
A type of reverse inequality can be deduced from \cite[Thm.~3.2]{Hinman23}, which says that for any $d$-dimensional polytope $P$ and $j\in\{0,\ldots,d-1\}$ one has that
\begin{equation}\label{eq:HinmanIneq}
    f_j(P) \geq \rho(d,j)f_{d-1}(P)\qquad\text{with}\qquad \rho(d,j):=\frac{1}{2}\left[\binom{\lceil \frac{d}{2}\rceil}{d-j-1}+\binom{\lfloor\frac{d}{2}\rfloor}{d-j-1}\right].
\end{equation}
However, it holds that $\rho(d,j)>0$ only for $j\geq\lfloor{d\over 2}\rfloor -1$, while $\rho(d,j)=0$ for all smaller values of $j$, see also Remark \ref{rem:Fallj=0} below. Putting $c_d:=\min_{j\in\{\lfloor{d\over 2}\rfloor -1,\ldots,d-1\}}\rho(d,j)\geq 1/2$, this yields
\begin{equation}\label{eq:fj>Cfd-1}
    \EE f_j(\cP_{n,\bd}^\bbeta) \geq {1\over 2}\EE f_{d-1}(\cP_{n,\bd}^\bbeta),
\end{equation}
for $j\in\{\lfloor{d\over 2}\rfloor -1,\ldots,d-1\}$. Combining \eqref{eq:fj<Cfd-1} with \eqref{eq:fj>Cfd-1} proves part (A) of Corollary \ref{cor:main}. \qed

\begin{remark}\label{rem:Fallj=0}
That we have to exclude the cases $j<\lfloor{d\over 2}\rfloor -1$ from Corollary \ref{cor:main} is not surprising. In fact, any linear inequality between $f_j(P)$ in this range for $j$ and $f_{d-1}(P)$ even for $d$-dimensional simplicial polytopes would contradict the upper-bound theorem in which the so-called cyclic polytopes appear as maximizers, see \cite[Thm.~8.23]{ZieglerBook}. This is in line with the observation that $\rho(d,j)=0$ in \eqref{eq:HinmanIneq} for $j<\lfloor{d\over 2}\rfloor -1$. We refer to \cite{Hinman23} for further discussions.
\end{remark}

\section{Open problems and conjectures}\label{sec:Discussion}

The results established in the present paper raise a number of natural questions, conjectures and point to interesting open problems, which might be addressed in future works. 

\smallskip

First, one might wonder if our results change if we replace in case of the uniform distribution the product of balls in the definition of $Z_\bd$ by a product of general convex bodies having a reasonably smooth boundary. 

\begin{conjecture}\label{conj1}
    Let $m\in\NN$. For $i\in\{1,\dotsc,m\}$ let $K_i\subset \RR^{d_i}$ be a convex body whose boundary is twice differentiable with positive Gaussian curvature everywhere and such that $\vol_{d_i}(K_i) = \vol_{d_i}(B_2^{d_i})$. Consider the product $K=K_1\times\ldots\times K_m\subset \RR^d$ with $d=d_1+\ldots+d_m$, and let $\cP_{n,d}(K)$ with $n\geq d+1$ be the convex hull of $n$ independent and uniformly distributed random points in $K$. We conjecture that
    \begin{equation*}
        \EE f_{j}(\cP_{n,\bd}(K)) {\asymp}   \EE f_{j}(\cP_{n,\bd}(Z_{\bd})) \qquad \text{for $j\in\{0,\dotsc,d-1\}$,}
    \end{equation*}
    and
    \begin{equation*}
        \vol_d(K) - \EE \vol_{d}(\cP_{n,\bd}(K)) {\asymp} \vol_d(Z_{\bd}) - \EE \vol_{d}(\cP_{n,\bd}(Z_{\bd})).
    \end{equation*}
\end{conjecture}

Note that the case $m=1$ is describing the well-investigated behavior of a random polytope whose generating points are distributed independently and uniformly inside a smooth convex body, see, for example, \cite{BFH:2010, CTV:2021, GW:2018, GLT:2020, LRSY:2019, PSSW:2024, SY:2023}. Hence for $m\geq 2$ this problem would pose a significant extension of the theory of random polytopes inside convex containers.

To tackle this conjecture, one should aim to replace Lemma \ref{lem:ReductionLemma} and similar results by using local approximations of the form
    \begin{equation*}
        \left(\prod_{i=1}^m r_iB_2^{d_i}\right)\cap H^+(\bw,s)
        \subset  K\cap H^+(\bw,s) 
        \subset \left(\prod_{i=1}^m R_iB_2^{d_i}\right)\cap H^+(\bw,s),
    \end{equation*}
    where $0<r_i<R_i$ for $i\in\{1,\dotsc, m\}$ may be chosen with respect to $\bw=\sum_{i=1}^m v_i\bu_i\in\SS^{d-1}$. The conjecture can directly be verified for ellipsoids. Indeed, consider an arbitrary sequence of linear subspaces $(L_1,\dotsc,L_m)$ of $\RR^d$ such that $\mathrm{dim}\,L_i=d_i$ and $L_i\cap L_j=\{0\}$ for all $i\neq j$. Choosing ellipsoids $E_i\subset L_i$ we may define $Z_\bd^{\rm ell}=E_1+\dotsc+E_m$. Then there is a affine transformation $A:\RR^d\to\RR^d$ such that $A(L_i)=\RR^{d_i}$, $A(E_i)=B_2^{d_i}$ and $A(Z_\bd^{\rm ell})=Z_{\bd}$. In particular, $\det A$ is determined by the volumes of the ellipsoids $E_i$:
    \begin{equation*}
        \det A = \prod_{i=1}^m \frac{\vol_{d_i}(E_i)}{\vol_{d_i}(B_2^{d_i})}.
    \end{equation*}
    Taking $\bbeta=(0,\dotsc,0)$ in Theorem \ref{thm:main} and Corollary \ref{cor:main}, by affine invariance, we have that
    \begin{equation*}
        \EE f_{j}(\cP_{n,\bd}(Z_\bd^{\rm ell})) = \EE f_{j}(\cP_{n,\bd}(Z_{\bd})) \qquad \text{for $j\in\{0,\dotsc,d-1\}$,}
    \end{equation*}
    and
    \begin{equation*}
        \vol_d(Z_\bd^{\rm ell}) - \EE \vol_{d}(\cP_{n,\bd}(Z_\bd^{\rm ell})) = \left(\det A\right) \left[\vol_d(Z_{\bd})- \EE \vol_{d}(\cP_{n,\bd}(Z_{\bd}))\right].
    \end{equation*}
    In particular, if we normalize $E_1,\ldots,E_m$ in such a way that $\vol_{d_i}(E_i)=B_2^{d_i}$ for all $i\in\{1,\ldots,m\}$, this shows that Conjecture \ref{conj1} is valid for $K=Z_\bd^{\rm ell}$.

    \medskip

    Corollary \ref{cor:main} says that for general block parameters $\bbeta$ the expected number of $j$-faces of $\cP_{d,\bd}^\bbeta$ is of the same order as the expected facet number $\EE f_{d-1}(\cP_{d,\bd}^\bbeta)$ as long as $j\geq \lfloor{d\over 2}\rfloor$. We believe that this restriction of the domain for $j$ is in fact not necessary.

    \begin{conjecture}
    For all $j\in\{0,1,\ldots,d-2\}$ it holds that
    $$
    \EE f_{j}(\cP_{d,\bd}^\bbeta) \asymp \EE f_{d-1}(\cP_{d,\bd}^\bbeta).
    $$
    \end{conjecture}

    In view of the inequality $f_j(P)\geq\min\{f_0(P),f_{d-1}(P)\}$, which holds for any $d$-dimensional polytope $P$ according to the main result of \cite{Hinman23}, to prove the previous conjecture it would be sufficient to show that $\EE f_{0}(\cP_{d,\bd}^\bbeta) \asymp \EE f_{d-1}(\cP_{d,\bd}^\bbeta)$. In view of Efron's identity for random polytopes, the latter would in turn follow from the correct asymptotic order for the so-called $T$-functional of $\cP_{d,\bd}^\bbeta$ we will be asking for below.

    \medskip

    In view of Theorem \ref{thm:main} and Corollary \ref{cor:main} it is natural to ask for the precise constant for the asymptotic expected face numbers and the asymptotic expected volume difference. For simplicity, we formulate the problem for the expected number of facets only. 

    \begin{problem}\label{op:ELimit}
    Assume the set-up of Theorem \ref{thm:main}. Find the constant $c_{\bbeta,\bd}\in(0,\infty)$ which satisfies
    $$
    \lim_{n\to\infty}n^{-\frac{k_{\max}-1}{k_{\max}+1}} (\ln n)^{-\#k_{\max}+1}\EE f_{d-1}(\cP_{n,\bd}^{\bbeta}) = c_{\bbeta,\bd}.
    $$
    How does this constant depend on the model parameters $\bbeta$ and $\bd$?
    \end{problem}

    Even if one restricts to the case of the uniform distribution, which corresponds to the choice $\bbeta=(0,\ldots,0$), finding a convenient description of $c_{\bbeta,\bd}$ seems a rather challenging task. A reason for this can be found in Step 4 of the proof of Theorem \ref{thm:main} presented in Section \ref{subsec:Conclusion}. Whereas the first step of the proof of Theorem \ref{thm:main} allow us to restrict the further analysis of the expected facet number to the level of the meta-cube, the steps 2 and 4 of the proof uncovers that the main contribution to $\EE f_{d-1}(\cP_{n,\bd}^{\bbeta})$ might -- depending on the beta-adjusted dimensions -- not only be cause by the contributions of the vertices of the meta-cube as one would expect from similar considerations for classical random polytopes as in \cite{BaranyBuchta,GusakovaReitznerThaele,Reitzner2005}. To determine $c_{\bbeta,\bd}$ it would be necessary to gain precise control over the further configurations that contribute to $\EE f_{d-1}(\cP_{n,\bd}^{\bbeta})$. We remark that if $m=1$, $\bd=(d)$ and $\bbeta=(\beta)=(0)$ the value of the constant $c_{\bbeta,\bd}$ is known to be
    $$
    c_{\bbeta,\bd} = {2\pi^{d(d-1)\over 2(d+1)}\over(d+1)!}{\Gamma(1+{d^2\over 2})\Gamma({d^2+1\over d+1})\over\Gamma({d^2+1\over 2})}(d+1)^{d^2+1\over d+1}\Big({\Gamma({d+1\over 2})\over\Gamma(1+{d\over 2})}\Big)^{d^2+1\over d+1},
    $$
    see \cite[Cor.\ 3]{Affentranger:1991}.

    \medskip

    In our main results, presented in the introduction, we have restricted ourself to the range in which the block-parameter $\bbeta$ satisfies $\bbeta\in[0,\infty)^m$, although the beta- and the block-beta distribution are well defined in the larger range $\bbeta\in(-1,\infty)^m$. Even further, many of our auxiliary geometric estimates developed in Sections \ref{sec:GeometricIngredients} remain valid in this full rage of parameters. However, in some of our computations were were not able to tackle the situation in which (some or all) our beta-parameters are between $-1$ and $0$. The main difficulty in this regime is that the corresponding beta-density is unbounded. Also note that for $\beta_i\in(-1,0)$ the beta-adjusted dimension $k_i=(d_i+\beta_i)/(1+\beta_i) > d_i$ is no longer bounded from above.

    \begin{conjecture}
    The results of Theorem \ref{thm:main} and Corollary \ref{cor:main} hold for $\bbeta\in(-1,\infty)^m$.
    \end{conjecture}

    A verification of this conjecture would be of particular interest in conjunction with a solution to Problem \ref{op:ELimit}. The reason for this is the fact that as $\beta\to -1$ the beta-distribution in the $d$-dimensional unit ball $B_2^d$ weakly converges to the uniform distribution on the $(d-1)$-dimensional unit sphere $\mathbb{S}^{d-1}$. Moreover, as $\beta\to\infty$, the suitably rescaled beta-distribution on $B_2^d$ weakly converges to the standard Gaussian distribution on $\mathbb{R}^d$. Studying the constant $c_{\bbeta,\bd}$ in such limiting regimes would potentially give insights into the geometric and combinatorial structure of block-polytopes where (some or all) generating random points are sampled with respect to a distribution with is either singular with respect to the Lebesgue measure (uniform distribution on $\mathbb{S}^{d-1}$) or has unbounded support (Gaussian distribution).

    \medskip

    For a $d$-dimensional polytope $P$ and real numbers $a$ and $b$, the $T$-functional of $P$ is defined as
    $$
    T_{a,b}(P) := \sum_{F\in\cF_{d-1}(P)}h(F)^a\vol_{d-1}(F)^b,
    $$
    where the sum is taken over all facets $F$ of $P$ and $h(F)$ denotes the distance of $F$ to the origin. The $T$-functional comprises a number of interesting geometric parameters associated with $P$. For example, $T_{0,0}(P)=f_{d-1}(P)$ is the number of facets, $T_{0,1}(P)$ is the surface area, ${1\over d}T_{1,1}(P)$ is the volume of $P$ if $P$ contains the origin, and $T_{1-p,1}(P)$ for $p\in\RR$ is the so-called $L_p$ surface area of $P$, see \cite{HLRT:2024} and the references therein.

    \begin{problem}
        Determine the asymptotic order of $\EE T_{a,b}(\cP_{n,\bd}^\bbeta)$, as $n\to\infty$. How does it depend on the parameters $a$ and $b$?
    \end{problem}

    A solution to this problem would be interesting, as it would allow for general $\bbeta$ to deal with the asymptotic expected volume of $\cP_{n,\bd}^\bbeta$ as well as with its surface area. Currently, the expected volume is only available if $\bbeta=(0,\ldots,0)$, see Corollary \ref{cor:main}. Of course, a sharpening in the spirit of Problem \ref{op:ELimit} would be of interest as well.

    \medskip

    In the present paper we were dealing mainly with the asymptotic order of the \textit{expected} facet number of the random polytopes $\cP_{n,\bd}^\bbeta$. For a deeper understanding of the random variable $f_{d-1}(\cP_{n,\bd}^\bbeta)$ it would be of interest to investigate, for example, its second-order properties and its fluctuation behavior.

    \begin{problem}
    Determine the variance asymptotics $f_{d-1}(\cP_{n,\bd}^\bbeta)$ and prove a central limit theorem together with a rate of convergence for this sequence of random variables, as $n\to\infty$.     
    \end{problem}
    
    We assume that the asymptotic order of the variance of $f_{d-1}(\cP_{n,\bd}^\bbeta)$ may be tackled by a more refined analysis of the geometry of the meta-cube. It would be interesting to see if, for example in the case $m=2$ corresponding to a product consisting of precisely two factors such as $Z_\bd=B_2^k\times B_2^{d-k}$, the Lagrangian products again play a distinguished role in the variance asymptotics. We expect this to be the case and also conjecture that this effect is be visible in the rate of convergence in the corresponding central limit theorem.

    \medskip

    In this paper we constructed $\cP_{n,\bd}^\bbeta$ as the convex hull of $n\geq d+1$ independent and identically distributed random points in $Z_\bd$ and found the asymptotic growth of $\EE f_{d-1}(\cP_{n,\bd}^\bbeta)$, as $n\to\infty$. It is natural to expect that this result continues to hold if the deterministic number $n$ of points is replaced by a random number $N$ of points, where $N$ is Poisson distributed with mean $n$.

    \begin{conjecture}
    For $n\in\NN$ let $N_n$ be a Poisson random variable with parameter $n$. Then
    $$
    \EE f_{d-1}(\cP_{N,\bd}^\bbeta) \asymp \EE f_{d-1}(\cP_{n,\bd}^\bbeta).
    $$
    \end{conjecture}

    Although such an additional randomization might seem artificial at first sight, it is on the other hand rather natural (and in some cases even necessary) to work with such a Poissonized model when dealing with the central limit theorem for the facet number. The main reason behind is the strong independence property of Poisson point processes, which makes the Possonized model typically easier to work with, see \cite{GusakovaReitznerThaele,ReitznerCLT} for example.

\subsection*{Acknowledgement}

AG and CT have been supported by the DFG priority program SPP 2265 \textit{Random Geometric Systems}. CT was also supported by the DFG priority program SPP 2458 \textit{Combinatorial Synergies}. AG was also supported by the DFG under Germany's Excellence Strategy  EXC 2044 -- 390685587, \textit{Mathematics M\"unster: Dynamics - Geometry - Structure}. Parts of this paper were written when the authors were participants of the Dual Trimester Program \textit{Synergies between modern probability, geometric analysis and stochastic geometry} at the Hausdorff Research Institute for Mathematics, Bonn. All support is gratefully acknowledged.

\bibliographystyle{plain}
\bibliography{References}

\end{document}